\documentclass[reqno,centertags,12pt]{amsart}
    \usepackage{amsmath,amsthm,amscd,amssymb,latexsym,verbatim}
    \usepackage{amssymb}

    \usepackage{graphicx,epsf,cite}
    
    \usepackage[shortlabels]{enumitem}
    
    \usepackage{xcolor}
    
    -\marginparwidth \oddsidemargin -4mm \evensidemargin \oddsidemargin

    \newtheorem{theorem}{Theorem}[section]

    \newtheorem{lemma}[theorem]{Lemma}
    
    \theoremstyle{definition}
    
    \newtheorem{example}[theorem]{Example}

    \newcounter{smalllist}

\DeclareMathOperator{\Var}{{Var}}

\DeclareMathOperator{\Cov}{{Cov}}

    \allowdisplaybreaks
    \numberwithin{equation}{section}
    
    \newcommand{\lb}{\label}

    \newcommand{\beq}{\begin{equation}}
    \newcommand{\eeq}{\end{equation}}
    
    \newcommand{\bal}{\begin{align}}
    \newcommand{\eal}{\end{align}}
    \newcommand{\bals}{\begin{align*}}
    \newcommand{\eals}{\end{align*}}
    
    \newcommand{\bbN}{{\mathbb{N}}}
    \newcommand{\bbR}{{\mathbb{R}}}
    
    \newcommand{\bbP}{{\mathbb{P}}}
    \newcommand{\bbE}{{\mathbb{E}}}
    \newcommand{\bbZ}{{\mathbb{Z}}}
    
    \newcommand{\bbQ}{{\mathbb{Q}}}
    
    \newcommand{\bbS}{{\mathbb{S}}}

    \newcommand{\calF}{{\mathcal F}}

    \newcommand{\calG}{{\mathcal G}}

    \newcommand{\eps}{\varepsilon}

\begin{document}
    \title[Time-dependent Subadditive Theorems]
    {Subadditive Theorems in Time-Dependent Environments}
    
    \author{Yuming Paul Zhang and Andrej Zlato\v s}
    
    \address{\noindent Department of Mathematics \\ University of
    California San Diego \\ La Jolla, CA 92093 \newline Email: \tt
    yzhangpaul@ucsd.edu, zlatos@ucsd.edu}
    
    
    
    \begin{abstract} 
    We prove time-dependent versions of Kingman's subadditive ergodic theorem, which can be used to study stochastic processes as well as propagation of solutions to PDE in time-dependent environments.
    \end{abstract}
    
    \maketitle

    \section{Introduction and Main Results} \lb{S1}
    
During the last half-century, Kingman's subadditive ergodic theorem \cite{kingman} and its versions (in particular, by Liggett \cite{Lig}) have been a crucial tool in the study of evolution processes in stationary ergodic environments, including first passage percolation and related models as well as processes modeled by partial differential equations (PDE) which satisfy the maximum principle.  Typically, the theorem is used to show that propagation of such a process in each spatial  direction has almost surely some deterministic asymptotic speed.  This can also often be extended to existence of a deterministic asymptotic propagation shape when the propagation involves invasion of one state of the process (e.g., the region not yet affected by it) by another (e.g., the already affected region).

Kingman's theorem concerns a family $\{X_{m,n}\}$ ($n>m\ge 0$) of random variables on a probability space which satisfy the crucial subadditivity hypothesis
\beq\lb{1.0}
X_{m,n} \leq X_{m,k}+X_{k,n} \qquad \text{for all $k\in\{m+1,\dots,n-1\}$},
\eeq
together with $\bbE[X_{0,n}]\in[-Cn,\infty)$ for some $C\ge 0$ and each $n\in\bbN$.  Also, $\{X_{m,n}\}$ is stationary in the sense that the joint distribution of $\{X_{m+n,m+n+k}\,|\,(n,k)\in\bbN_0 \times \bbN\}$ is independent of $m\in\bbN_0$.  It  then concludes that $X:=\lim_{n\to \infty}\frac{X_{0,n}}{n}$ exists almost surely, and 
\[
\bbE[X]= \lim_{n\to \infty}\frac{\bbE\left[X_{0,n}\right]}{n}
 =  \inf_{n\ge 1}\frac{\bbE\left[  X_{0,n} \right]}{n}.
\]
Moreover, $X$ is a constant if $\{X_{m,n}\}$ is also ergodic, that is, any event defined in terms of $\{X_{m,n}\}$ and invariant under the shift $(m,n)\mapsto(m+1,n+1)$ has probability either 0 or 1.

A typical use of such a result in the study of PDE is described in Example \ref{E.4.3} below.
We let $X_{m,n}$ be the time it takes for a solution to the PDE to propagate from $me\in\bbR^d$ to $ne\in\bbR^d$ (see the example for details), with $e$ some fixed unit vector (i.e., direction).  Subadditivity is then guaranteed by the maximum principle for the PDE, and Kingman's theorem may therefore often be used to conclude existence of a deterministic propagation speed in direction $e$, in an appropriate sense and under some basic hypotheses.  

However, this approach only works when the coefficients of the PDE are either independent of time or time-periodic.  The present work is therefore  motivated by our desire to apply subadditivity-based techniques to PDE with more general time dependence of coefficients (and to other non-autonomous models), in particular, those with finite temporal ranges of dependence as well as with decreasing temporal correlations.  Despite this being a very natural question, we were not able to find relevant results in the existing literature.  We thus prove here the following two results, and also provide applications to a time-dependent first passage percolation model (see Examples \ref{E.5.1} and \ref{E.5.2} below).  In the companion paper \cite{ZhaZla5} we apply these results to specific PDE models (as described in Example \ref{E.4.3}), specifically  reaction-diffusion equations and Hamilton-Jacobi equations.

Our first main result in the present paper applies when the  process in question (or rather the environment in which it occurs) has a finite temporal range of dependence, with $\calF_t^\pm$ being the sigma-algebras generated by the environment up to and starting from time $t$, respectively.  It mirrors Kingman's theorem, with a weaker stationarity hypothesis (3) below (analogous to \cite{Lig}) but under the additional hypothesis (6).  The latter is the natural requirement that if the process propagates from some ``location'' $m$ to another location $n$, starting at some time $t$, it cannot reach $n$ later than the same process which starts form $m$ at some later time $t+s$, at least when $s$ is sufficiently large.  In the case of PDE, maximum principle will often guarantee this if the time-dependent  propagation times $X_{m,n}^t\ge 0$ (i.e., from location $m$ to $n$, starting at time $t\in[0,\infty)$)  are defined appropriately (see Example \ref{E.4.3}).  We also note that (1) below is the natural version of \eqref{1.0} in the time-dependent setting.
    
\begin{theorem} \lb{T.1.1}   
Let $(\Omega,\bbP,\calF)$ be a probability space, and $\{\calF^{\pm}_t\}_{t\geq 0}$  two filtrations such that 
\[
\calF^{-}_s\subseteq\calF^{-}_{t}\subseteq \calF  \qquad\text{and}\qquad  \calF\supseteq \calF^{+}_{s}\supseteq\calF^{+}_{t}
\]
for all $t\geq s\geq 0$.
For any $t\ge 0$ and integers $n>m\geq 0$, let $X_{m,n}^t:\Omega\to [0,\infty)$ be a random variable.
Let there be $C\ge 0$ such that the following statements hold for all such $t,m,n$.
    \begin{enumerate}[(1)]
        \item $X_{m,n}^t \leq X_{m,k}^t+X_{k,n}^{t+X_{m,k}^t}$ for all $k\in\{m+1,\dots,n-1\}$;
        
        \smallskip
        
        \item $\bbE \left[X_{0,1}^0\right]<\infty$;
        
        \smallskip
        
        \item the joint distribution of $\{X_{m,m+1}^t, X_{m,m+2}^t, \dots \}$ is independent of $(t,m)$;
        
        \smallskip
        
        \item   $X_{m,n}^t$ is $\calF^+_t$-measurable, and 
        $\{\omega\in\Omega\,|\, X_{m,n}^t(\omega)\le s\}\in \calF^{-}_{t+s}$ for any $s\ge 0$; 
        
         \smallskip
        
        \item  $\calF^{-}_t$ and $\calF_{t+C}^+$ are independent; 
        
        \smallskip

        \item $X_{m,n}^t \leq X_{m,n}^{t+s}+s$ for all $s\in [C,C+c]$, with some $c>0$.
    \end{enumerate}
Then 
\beq\lb{4.1}
\lim_{n\to \infty}\frac{X_{0,n}^0 }{n}= \lim_{n\to \infty}\frac{\bbE\left[X_{0,n}^0\right]}{n}
 =  \inf_{n\ge 1}\frac{\bbE\left[  X_{0,n}^0 \right]+C}{n} 
\qquad\text{ almost surely.}
\eeq
Moreover,  if $C\in\bbN$ and $X_{m,n}^t$ are all integer-valued, then it suffices to have $c=0$ in (6).
\end{theorem}    
    
{\it Remarks.}   1.  Of course, it suffices to assume  (1) and (6) only almost surely.
\smallskip

2. There would be little benefit in using different $C$ in (5) and (6) because (5) clearly holds with any larger $C$, while iterating (6) yields (6)  for all $s\in[kC,kC+kc]$ and any $k\in\bbN$.
\smallskip

3.  The ergodicity hypothesis in \cite{kingman} is here replaced by (5) (or by ($5^{*}$) below).
\smallskip

Our second main result  allows for an infinite temporal range of dependence of the environment, provided this dependence decreases with time in an appropriate sense, and we then also need  a uniform bound in place of (2).

\begin{theorem} \lb{T.1.2}
Assume the hypotheses of Theorem \ref{T.1.1}, but with
(2) and (5) replaced by 
\begin{enumerate}
\item[($2^{*}$)]  $X_{0,1}^0\le C$;

\smallskip

\item[($5^{*}$)]  $\lim_{s\to\infty}\phi(s)=0$, where 
\[
\phi(s):=\sup\left\{\left| \bbP[F|E] -\bbP[F]\right| \,\big|\,\, t\ge 0 \,\,\&\,\, (E,F)\in \calF_t^- \times \calF_{t+s}^+ \,\,\&\,\, \bbP[E]>0 \right\}.
\]
\end{enumerate}
Then 
\beq\lb{306}
\lim_{n\to\infty}\frac{X_{0,n}^0}{n} =\lim_{n\to \infty}\frac{\bbE\left[X_{0,n}^{0}\right]}{n}\qquad\text{ in probability},
\eeq
and if there is $\alpha>0$ such that $\lim_{s\to\infty} s^\alpha \phi(s)=0$, then also
\beq\lb{4.1'}
\lim_{n\to \infty}\frac{X_{0,n}^0 }{n}= \lim_{n\to \infty}\frac{\bbE\left[X_{0,n}^0\right]}{n}
\qquad\text{ almost surely.}
\eeq
Moreover,  if $C\in\bbN$ and $X_{m,n}^t$ are all integer-valued, then it suffices to have $c=0$ in (6).
\end{theorem}

{\it Remarks.}  
1. Again, using different $C$ in ($2^*$) and (6) would not strengthen the result.
\smallskip

2. We will actually prove this result with $\phi(s)$ being instead the supremum of
\[
\sum_{i\ge 0} \left|\bbP[F_i\cap E_i] - \bbP[F_i]\bbP[E_i] \right| 
\]
over all $\{(E_i,F_i)\in \calF_{t_i}^- \times \calF_{t_i+s}^+\}_{i\ge 0}$ with $t_0,t_1,\dots\ge 0$ and $E_0,E_1,\dots$  pairwise disjoint (which is clearly no more than $\phi(s)$ from ($5^{*}$)).
\smallskip

3. We will also show that without  assuming $\lim_{s\to\infty} s^\alpha \phi(s)=0$, we still have
\beq\lb{4.1''}
\liminf_{n\to \infty}\frac{X_{0,n}^0 }{n}\ge \lim_{n\to \infty}\frac{\bbE\left[X_{0,n}^0\right]}{n}
\qquad\text{ almost surely.}
\eeq
\smallskip

{\bf Organization of the Paper and Acknowledgements.}
We prove Theorem \ref{T.1.1} in Section \ref{S2} and the claims in Theorem \ref{T.1.2} in Sections \ref{S3} and \ref{S3'}.  Section \ref{S4} contains some applications of these results.

We thank Patrick Fitzsimmons and Robin Pemantle for useful discussions.
YPZ acknowledges partial support by an AMS-Simons Travel Grant.  AZ acknowledges partial support by  NSF grant DMS-1900943 and by a Simons Fellowship.

\section{Finite Temporal Range of Dependence} \lb{S2}

Let us first prove a version of Theorem \ref{T.1.1} with $\bbN_0$-valued random variables and $C=0$ in (5).  Theorem \ref{T.1.1} will then easily follow. Let us denote $\{X=s\}:=\{\omega\in\Omega\,|\, X(\omega)=s\}$.

\begin{theorem}    \lb{T.2.5}
Let $(\Omega,\bbP,\calF)$ be a probability space, and $\{\calF^{\pm}_t\}_{t\in\bbN_0}$  two filtrations such that 
\beq \lb{4.5}
\calF^{-}_0\subseteq\calF^{-}_{1}\subseteq\dots\subseteq \calF  \qquad\text{and}\qquad  \calF\supseteq \calF^{+}_{0}\supseteq\calF^{+}_{1}\supseteq \dots.
\eeq
For any integers $t\ge 0$ and $n>m\geq 0$, let $T_{m,n}^t:\Omega\to \bbN_0$ be a random variable.
Let there be $C,C'\in \bbN$ such that the following statements hold for all such $t,m,n$.
    \begin{enumerate}[(1')]
        \item $T_{m,n}^t\leq T_{m,k}^t+T_{k,n}^{t+T_{m,k}^t}$ for all $k\in\{m+1,\dots,n-1\}$;
        
        \smallskip
        
        \item $\bbE \left[T_{0,1}^0\right]\leq C'$;
        
        \smallskip
        
        \item  the joint distribution of $\{T_{m,m+1}^t, T_{m,m+2}^t, \dots\}$  is independent of $(t,m)$;

        \smallskip
        
        \item $T_{m,n}^t$ is $\calF^+_t$-measurable, and 
        $\{T_{m,n}^t=j\}\in \calF^{-}_{t+j}$ for any $j\in\bbN_0$; 
        
         \smallskip
        
        \item  $\calF^{-}_t$ and $\calF_{t}^+$ are independent;
        
        \smallskip
        
         \item $T_{m,n}^t\leq T_{m,n}^{t+C}+C$.
    \end{enumerate}
Then    
\beq\lb{223}
\lim_{n\to \infty}\frac{T_{0,n}^0 }{n}= \lim_{n\to \infty}\frac{\bbE\left[T_{0,n}^0\right]}{n} = \inf_{n\ge 1}\frac{\bbE\left[T_{0,n}^0\right]}{n}
 \qquad\text{ almost surely.}
 \eeq
\end{theorem}

\begin{proof}


First, we claim that  almost surely we have
 \beq\lb{211}
    \limsup_{n\to \infty}\frac{T_{0,n}^0}{n}\leq \lim_{n\to \infty}\frac{\bbE\left[T_{0,n}^0\right]}{n}=\inf_{n\ge 1}\frac{\bbE\left[T_{0,n}^0\right]}{n}.
\eeq
The proof of \eqref{211} is similar to the proof of \cite[Lemma 6.7]{BIN}, although there the analogs of $T_{m,n}^t$ were bounded random variables; the idea goes back to \cite{kingman}, where the analogs of $T_{m,n}^t$ were $t$-independent.
For any integers $n>m> 0$, (4') shows that for any $i,j\in\bbN_0$ we have
\[
\{T_{0,m}^0 =i\}\in \calF^-_{i}\qquad\text{and}\qquad \{T_{m,n}^{i} =j\}\in \calF^+_{i}.
\]
Therefore (5') and (3') yield
\[
\bbP\left[T_{0,m}^0 =i \,\, \& \,\, T_{m,n}^i =j\right]
=\bbP\left[T_{0,m}^0 =i\right]\bbP\left[ T_{m,n}^i =j\right]
 =\bbP\left[T_{0,m}^0 =i\right]\bbP\left[ T_{0,n-m}^0 =j\right].
\]
Summing this over $i\in\bbN_0$, we find that $T_{m,n}^{T_{0,m}^0}$ ($=T_{m,n}^{T_{0,m}^0(\cdot)}(\cdot)$) has the same distribution as $T_{0,n-m}^0$.
Thus from (1') we obtain
\[
\bbE\left[T_{0,n}^0 \right]\leq \bbE\left[T_{0,m}^0 \right]+\bbE\left[T_{m,n}^{T_{0,m}^0 } \right]\leq \bbE\left[T_{0,m}^0 \right]+\bbE\left[T_{0,n-m}^0 \right].
\]
Fekete's subadditive lemma thus implies that the equality in \eqref{211} holds.

For any $n\in\bbN$, let $t^{n}_0:=0$ and $\xi_0^{n}:=T_{0,n}^0$, and then for $i\in\bbN$ define recursively
\[
 t^{n}_i :=t^{n}_{i-1} +\xi_{i-1}^{n} \qquad\text{and}\qquad  \xi_i^{n} :=T_{in,(i+1)n}^{t^{n}_{i} } .
\]
By iteratively applying (1'), we get for any $k\in\bbN$,
\beq\lb{212}
T_{0,kn}^0 \leq \sum_{i=0}^{k-1} \xi_i^{n}.
\eeq
Similarly as above, it follows from (3')--(5') that for any $j_0,j_1,\dots,j_{k-1}\in \bbN_0$ we have
\begin{align*}
    \bbP \left[\xi^{n}_i =j_i\text{ for  }i=0,\dots, k-1\right] 
&=\bbP\left[\xi^{n}_i =j_i\text{ for  }i=0,\dots, k-2 \,\, \& \,\, T_{(k-1)n,kn}^{\sum_{i=0}^{k-2} j_i} =j_{k-1}\right]\\ 
&=\bbP\left[\xi^{n}_i =j_i\text{ for  }i=0,\dots, k-2\right] \, \bbP\left[T_{(k-1)n,kn}^{\sum_{i=0}^{k-2} j_i} =j_{k-1}\right]\\ 
&=\bbP\left [\xi^{n}_i =j_i\text{ for  } i=0,\dots, k-2\right] \, \bbP\left[T_{0,n}^0 =j_{k-1}\right]\\
&= \ldots = \prod_{i=0}^{k-1} \bbP\left[T_{0,n}^0 =j_i\right].   
\end{align*}
Summing this over all indices but $i$ shows that $\xi_i^{n}$ has the same law as $T_{0,n}^0$ for each $i$. This, (2'), and \eqref{212} with $n=1$ then show that for any $k\in\bbN$, 
\beq\lb{216}
\bbE\left[T_{0,k}^0\right]\leq \sum_{i=0}^{k-1}\bbE\left[\xi_i^1\right]=k\bbE\left[T_{0,1}^0\right]\leq C'k.
\eeq
Also, the above computation shows that $\xi_0^{n},\dots, \xi_{k-1}^n$ are jointly independent random variables  for all $n$ and $k$, so the strong law of large numbers yields
\[
\lim_{k\to\infty}\frac{1}{k}\sum_{i=0}^{k-1}\xi_i^{n} =\bbE\left[T_{0,n}^0 \right] \qquad\text{ almost surely}.
\]
Thus \eqref{212} and the equality in \eqref{211} yield that for any $\eps>0$ there is $n_\eps\in\bbN$ such that
\beq\lb{219}
\limsup_{k\to \infty}\frac{T_{0,kn_\eps}^0 }{k n_\eps}\leq \frac{\bbE\left[T_{0,n_\eps}^0 \right]}{n_\eps}\leq (1+\eps)\lim_{n\to \infty}\frac{\bbE\left[T_{0,n}^0 \right]}{n} \qquad\text{ almost surely}.
\eeq
Now fix any $l\in\{0,\dots,n_\eps-1\}$
and note that (1') yields for all $k\in\bbN_0$,
\beq\lb{217}
T_{0,kn_\eps+l}^0\leq T_{0,kn_\eps}^0+T_{kn_\eps,kn_\eps+l}^{T_{0,kn_\eps}^0}.
\eeq
Since $T_{kn_\eps,kn_\eps+l}^{T_{0,kn_\eps}^0}$ has the same distribution as $T_{0,l}^0$, we obtain from \eqref{216} that
\[
\sum_{k\geq 0}\bbP\left[T_{kn_\eps,kn_\eps+l}^{T_{0,kn_\eps}^0}>(kn_\eps+l)\eps \right]\leq \sum_{k\geq 0}\bbP\left[T_{0,l}^{0}>k\eps \right]\leq\frac{1}{\eps}\bbE\left[T_{0,l}^0\right]<\infty.
\]
Borel-Cantelli Lemma then implies that $\limsup_{k\to\infty}\frac{1}{kn_\eps+l}T_{kn_\eps,kn_\eps+l}^{T_{0,kn_\eps}^0}\leq \eps$ almost surely.
This and \eqref{217} for each $l\in\{0,\dots,n_\eps-1\}$, together with \eqref{219}, now show that 
\[
\limsup_{n\to\infty}\frac{T_{0,n}^{0}}n\leq \eps+ (1+\eps)\lim_{n\to \infty}\frac{\bbE\left[T_{0,n}^0 \right]}{n}  \qquad\text{ almost surely}.
\]
Taking $\eps\to 0$ now yields the inequality in \eqref{211}.

Next, for each $(t,m)\in\bbN_0^2$ let
    \[
    Z_m^t :=\liminf_{n\to \infty}\frac{T_{m,m+n}^t }{n}.
    \]
It follows from (6') that $Z_m^{t+Ck}$ is non-decreasing in $k\in\bbN$. But since the law of $Z_m^t$ is independent of $(t,m)$ by (3'), we must   almost surely have $Z_m^{t+Ck}=Z_m^t$ for all $k\in\bbN$. However, this and (3') imply that $Z_m^t $ is independent of $\calF^{-}_{t+Ck}$ for all $k\in\bbN$, while (4') shows that it is also measurable with respect to the $\sigma$-algebra generated by $\bigcup_{s\ge t} \calF^{-}_s$. This shows that there is a constant $Q\in [0,\infty)$ such that $Z_m^t=Q$ almost surely for each $(t,m)\in\bbN_0^2$.

In view of \eqref{211}, 
to prove \eqref{223} it remains to show that
    \beq\lb{110}
    Q\geq \lim_{n\to\infty}\frac{\bbE \left[T_{0,n}^0 \right]}{n}.
    \eeq
Our proof of this is related to the approach of Levental \cite{lev} in the $t$-independent case, which is in turn  based  on  \cite{KW}.  However, $t$-dependence complicates the situation here, which is why we first needed to show that $Z_m^t$ is in fact $(t,m,\omega)$-independent to conclude \eqref{110} (in \cite{lev}, it was sufficient to allow $\omega$-dependence at first).
Fix any $\eps>0$, and denote $Q_\eps:=Q+\eps$ and 
    \[
    N^{t}_m :=\min\left\{n\geq 1\,\big|\, T_{m,m+n}^t \leq n\, Q_\eps\right\} 
    \]
(which also depends on $\eps$ but we suppress this in the notation).
It follows from $Z_m^t=Q$ a.e.~that  almost surely we have $N^{t}_m<\infty$ for all $(t,m)\in\bbN_0^2$, and (3') yields that $N^{t}_m $ has the same distribution as $N^{0}_0 $. Moreover, $N_m^{t}$ is $\calF^{+}_t$-measurable   by (4').
Next, let $M_\eps\in [1,\infty)$  be a large constant such that
    \beq\lb{112}
    \bbE \left[T_{0,1}^0 1_{\left\{N^{0}_0 > M_\eps\right\}}\right]\leq \eps.
    \eeq

    Let now $t_0:=0$ and $r_0:= 0$, and for $k\geq 0$ define recursively
    \[
    r_{k+1}:=r_{k}+N^{t_{k}}_{r_{k}}1_{\left\{N^{t_{k}}_{r_{k}}\leq M_\eps\right\}} + 1_{\left\{N^{t_{k}}_{r_{k}}>M_\eps\right\}} 
    \qquad\text{and}\qquad    t_{k+1}:=t_{k}+T_{r_{k},r_{k+1}}^{t_{k}}.
    \]
Fix any $n\in\bbN$.  We will now use $\{r_k\}_{k\ge 1}$ to divide the ``propagation'' from $0$ to $n$ into several ``steps''.  Since this sequence is strictly increasing  for each $\omega\in\Omega$, the random variable
\[
K_n:=\min\{k\in\bbN_0 \,|\,r_k \ge n-M_\eps\}
\]
is well defined, and  satisfies $0\le K_n\leq n-1$ and $r_{K_n}\in[n-M_\eps, n-1]$. Then (1') yields
    \beq\lb{114}
    T_{0,n}^0 \leq \sum_{k=0}^{K_n -1} T_{r_{k} ,r_{k+1} }^{t_{k} } +T_{r_{K_n } ,n}^{t_{K_n } } =: S_n + T_{r_{K_n } ,n}^{t_{K_n } }
    \eeq
(note that, e.g., $T_{r_{K_n } ,n}^{t_{K_n } } = T_{r_{K_n(\cdot) }(\cdot),n}^{t_{K_n(\cdot)}(\cdot)}(\cdot)$).
If $N^{t_k}_{r_k}\leq M_\eps$, then
    \[
 T_{r_{k},r_{k+1}}^{t_{k}}\leq (r_{k+1}-r_{k})\, Q_\eps,
    \] 
while  if $N^{t_k}_{r_k}> M_\eps$, then $r_{k+1}=r_k+1$.
Hence
we obtain
\beq\lb{115}
{S_n} \leq \sum_{k=0}^{K_n-1}
(r_{k+1}-r_{k})\, Q_\eps+\sum_{k=0}^{K_n-1} T_{r_k,r_{k}+1}^{t_{k}} 1_{\left\{N^{t_{k}}_{r_{k}}>M_\eps\right\}}
\le r_{K_n} Q_\eps+ \sum_{k=0}^{n-1} T_{r_k,r_{k}+1}^{t_{k}}1_{\left\{N^{t_{k}}_{r_{k}}>M_\eps\right\}}.
\eeq

We now want to take expectation on both sides of \eqref{115}.
From (4') we see that for any $i,j\in\bbN_0$ we have $\{r_{k}=i \,\,\&\,\, t_k=j\}\in \calF^{-}_{j}$. Since $T_{i,i+1}^j$ and $N^{j }_{i}$ are $\calF_j^+$-measurable, from (5'), (3'), and \eqref{112} we obtain
\begin{align*}
   \bbE \left[ T_{r_{k} ,r_{k} +1}^{t_{k} } 1_{\left\{N^{t_{k} }_{r_{k} } >M_\eps\right\}}\right]
&=\sum_{i,j \ge 0}\bbE\left[T_{i,i+1}^{j} 1_{\left\{N^{j }_{i} >M_\eps\right\}} 1_{\left\{r_{k}=i \,\, \& \,\, t_{k}=j \right\} }\right]\\
&=\,\sum_{i,j \ge 0}\bbE\left[ T_{i,i+1}^{j} 1_{\left\{N^{j}_{i}>M_\eps\right\}}\right]\bbP\left[ r_{k} =i \,\, \& \,\, t_{k} =j\right]\\
&=\,\bbE\left[ T_{0,1}^0 1_{\left\{N^{0}_{0}>M_\eps\right\}}\right]\leq \eps.
\end{align*}
So \eqref{115} and  $r_{K_n}\leq n$ yield
    \beq\lb{113}
    \frac{\bbE\left[S_n \right]}{n}\leq  \frac{\bbE [r_{K_n}]Q_\eps}{n} + \eps\leq Q+2\eps.
    \eeq
    
    Finally, we claim that $\bbE \left[T_{r_{K_n } ,n}^{t_{K_n }}\right]\le C'M_\eps^2$; this together with \eqref{114} and   \eqref{113}, and then taking $\eps\to 0$, will yield \eqref{110}. To this end we note that  $1\le n-r_{K_n} \le M_\eps$ implies
    \beq\lb{300}
T_{r_{K_n } ,n}^{t_{K_n }} \le \max_{l\in\{1,\dots, \min\{M_\eps,n\}\}} T_{n-l ,n}^{t_{K_n }}.
\eeq
Since $\{t_{K_n}=j\}\in \calF^{-}_{j}$ and $T_{n-l,n}^j$ is $\calF_j^+$-measurable, we obtain from (5'), (3'), and \eqref{216},
\[
   \bbE \left[ T_{n-l ,n}^{t_{K_n }} \right]
   =\sum_{j \ge 0}\bbE\left[ T_{n-l ,n}^{j} 1_{\left\{ t_{K_n }=j \right\} }\right] 
= \sum_{j \ge 0}\bbE\left[ T_{n-l ,n}^{j} \right]   \bbP\left[  t_{K_n }=j \right] 
=\bbE\left[ T_{0 ,l}^{0} \right] \le C'l.
\]
Therefore indeed 
\[
\bbE \left[T_{r_{K_n } ,n}^{t_{K_n }}\right] \le \sum_{l=1}^{M_\eps} C'l \le C'M_\eps^2,
\]
so \eqref{110} holds and the proof is finished.
\end{proof}

\begin{proof}[Proof of Theorem \ref{T.1.1}]
Let us first assume that $c\ge 1$ and  define 
\[
T_{m,n}^t:= \lceil X_{m,n}^t +C \rceil \qquad (\in\bbN_0).
\]
Let us redefine $\calF_t^-$ to be $\calF_{t-C}^-$ for $t\ge C$ and $\{\emptyset,\Omega\}$ for $t\in[0,C)$ (i.e., shift $\calF_t^-$ to the right by $C$)
 and let $C':=\bbE\left[\lceil X_{0,1}^0 +C \rceil\right]$.
After restricting $t$ to $\bbN_0$, it is clear that $T_{m,n}^t$ satisfies hypotheses (2')--(6') of Theorem \ref{T.2.5}, with $\max\{ \lceil C\rceil,1\}$ in place of $C$.
And  (1') also holds because if $n>k>m\geq 0$ are integers, then
($1$) and  ($6$) with $s:=\lceil X_{m,n}^t +C \rceil- X_{m,n}^t$ yield
\[
T_{m,n}^t = \lceil X_{m,n}^t +C \rceil \leq \left\lceil X_{k,n}^{t+X_{m,k}^t} + X_{m,k}^t +C \right\rceil
\leq \left\lceil X_{k,n}^{t+T_{m,k}^t}+T_{m,k}^t +C \right\rceil  \le  T_{k,n}^{t+T_{m,k}^t}+T_{m,k}^t.
\]
Hence \eqref{223} 
proves \eqref{4.1} with the last numerator being $\bbE\left[\lceil X_{0,n}^0 +C \rceil\right]$.  Note that this argument also applies in the setting of the last claim in Theorem \ref{T.1.1} and without $\lceil\cdot\rceil$.

To get \eqref{4.1} as stated and for any $c>0$, let $S\ge \frac 1c$,  $\calG_t^\pm:=\calF_{t/S}^\pm$, and $Y_{m,n}^t:=SX_{m,n}^{t/S}$.
Since the above argument applies with $(\calG_t^\pm,Y_{m,n}^t,SC, Sc)$ in place of $(\calF_t^\pm,X_{m,n}^t,C,c)$, we obtain \eqref{4.1} with the last numerator being $\bbE\left[\frac 1S\lceil S(X_{0,n}^0+C)\rceil\right]$.  Taking $S\to\infty$ yields \eqref{4.1}.
\end{proof}

\section{Time-Decaying Dependence I}  \lb{S3}

In this section we will prove the first claim in Theorem~\ref{T.1.2} and the corresponding integer-valued claim.  Let us first prove a version of the latter with weaker ($2^{*}$) and stronger ($5^{*}$).

\begin{theorem} \lb{T.2.1}
Let $(\Omega,\bbP,\calF)$ be a probability space, and $\{\calF^{\pm}_t\}_{t\in\bbN_0}$  two filtrations satisfying \eqref{4.5}.
For any integers $t\ge 0$ and $n>m\geq 0$, let $X_{m,n}^t:\Omega\to \bbN_0$ be a random variable.
Let there be $C\in \bbN$ such that  for all such $t,m,n$ we have (1) and (3) from Theorem \ref{T.1.1}, and
\begin{enumerate}
\item[($2^{**}$)] $\bbE \left[X_{0,1}^0\right]+\bbE \left[( X_{0,1}^0)^2\right]<\infty$;

        \smallskip
        
\item[($4^{**}$)] $X_{m,n}^t$ is $\calF^+_t$-measurable, and 
        $\{X_{m,n}^t=j\}\in \calF^{-}_{t+j}$ for any $j\in\bbN_0$; 
        
        \smallskip
        
\item[($5^{**}$)]  $\lim_{s\to\infty} \phi(s)=0$, where 
\[
\phi(s):=\sup\left\{\left|\frac{\bbP[F|E] }{\bbP[F]}-1\right| \,\bigg|\,\,  t\in\bbN_0 \,\,\&\,\, (E,F)\in \calF_t^- \times \calF_{t+s}^+ \,\,\&\,\, \bbP[E]\bbP[F]>0 \right\}.
\]

        \smallskip
        
         \item[($6^{**}$)]  $X_{m,n}^t\leq X_{m,n}^{t+C}+C$.  
\end{enumerate}
Then \eqref{306} holds.
\end{theorem}

\begin{proof}
From ($5^{**}$) we know that  for each $\eps>0$, there is $C_\eps\in \bbN$ which is a multiple of $C$ from ($6^{**}$) and
\beq\lb{301}
\phi( C_{\eps} )\le \eps.
\eeq
Let us then define (again suppressing $\eps$ in the notation for the sake of clarity)
\[
T_{m,n}^{t}:=X_{m,n}^t+C_{\eps}.  
\]
As before, one can easily check that (1), ($2^{**}$), (3), ($5^{**}$), and ($6^{**}$) still hold with $X_{m,n}^{t}$ replaced by $T_{m,n}^{t}$, and ($4^{**}$) can be replaced by
\begin{itemize}
    \item[($4$'')] $T_{m,n}^{t}$ is $\calF^+_t$-measurable, and 
        $\{ T_{m,n}^{t}=j\}\in \calF^{-}_{t+j-C_{\eps}}$ for any $j\in\bbN_0$.
\end{itemize}

Next, let
\beq\lb{319}
\underline{X}:=\liminf_{n\to \infty}\frac{\bbE\left[X_{0,n}^0\right]}{n}.
\eeq
As before, for any $\eps>0$ and $n\in\bbN$, let $t^{n}_{0}:=0$ and $\xi_0^{n}:=T_{0,n}^{0}$, and then for $i\in\bbN$ define recursively
\[
 t^{n}_{i} :=t^{n}_{i-1} +\xi_{i-1}^{n} ,\qquad   \xi_i^{n} :=T_{in,(i+1)n}^{t^{n}_{i} }, \qquad\text{and}\qquad \mu_i^{n}:=\bbE\left[\xi_i^{n}\right].
\]
By (1) we have $T_{0,kn}^{0}\leq\sum_{i=0}^{ k-1}\xi_i^{n}$ for each $k\in\bbN$. Also, since ($4$'') yields
\beq\lb{401}
\{t_i^{n} =k\}\in \calF^-_{k-C_{\eps}}\qquad\text{and}\qquad \{T_{in,(i+1)n}^{k} =j\}\in \calF^+_{k},
\eeq
it follows from (3)  and \eqref{301} that
\beq\lb{304}
\begin{aligned}
\mu_i^{n}&=\sum_{k,j\geq 0} j\,\bbP\left[T_{in,(i+1)n}^{k}=j \,\, \& \,\,  t_i^{n}=k\right]\leq (1+\eps)\sum_{k,j\geq 0} j\,\bbP\left[T_{in,(i+1)n}^{k}=j\right]\bbP\left[ t_i^{n}=k\right]\\
&=(1+\eps)\sum_{j\geq 0} j\,\bbP\left[T_{0,n}^{0}=j\right]=(1+\eps)\bbE\left[ T_{0,n}^{0}\right]=(1+\eps)\mu_0^{n}.
\end{aligned}
\eeq
We can similarly obtain
\beq\lb{309}
\bbE\left[\xi_i^{n}-C_{\eps}\right]\leq (1+\eps)(\mu_0^{n}-C_{\eps})
\eeq 
and
\beq\lb{309'}
\bbE\left[\xi_i^{n}-C_{\eps}\right]\geq (1-\eps)(\mu_0^{n}-C_{\eps})
\eeq 
because $\xi_i^n\ge C_\eps$.
Invoking (1) and \eqref{304} with $n=1$ yields for $C'_\eps:=\bbE[X_{0,1}^0]+C_\eps$,
\beq\lb{320}
\mu_0^{n}\leq \sum_{i=0}^{n-1}\mu_i^{1}\leq (1+\eps)n\mu_0^{1}\leq (1+\eps) C'_{\eps} n.
\eeq
This implies that $\underline{X} \le (1+\eps)C'_\eps$.

Let us now pick $n_\eps\in \bbN$  such that
\beq\lb{303}
\frac{\bbE\left[T_{0,n_\eps}^{0}\right]}{n_\eps}\leq \underline{X}+\eps,
\eeq
which exists by \eqref{319}.
Then ($4$'') shows that for any integers $n>m>0$ and $i,j\in\bbN_0$ we get
\[
\{T_{0,m}^{0} =i\}\in \calF^-_{i-C_{\eps}}\qquad\text{and}\qquad \{T_{m,n}^{i} =j\}\in \calF^+_{i}.
\]
Therefore (3) and \eqref{301} yield
\beq\lb{307}
\begin{aligned}
\bbE\left[T_{m,n}^{T_{0,m}^{0}}\right] &=\sum_{i,j\geq 0} j\,\bbP\left[T_{m,n}^{i}=j \,\, \& \,\,  T_{0,m}^{0}=i\right]\leq (1+\eps)\sum_{i,j\geq 0} j\,\bbP\left[T_{m,n}^{i}=j\right]\bbP\left[ T_{0,m}^{0}=i\right]\\
&=(1+\eps)\sum_{j\geq 0} j\,\bbP\left[T_{0,n-m}^{0}=j\right]=(1+\eps)\bbE\left[ T_{0,n-m}^{0}\right].
\end{aligned}
\eeq
For any $n\in\bbN$ write $n=kn_\eps+l$, where $k\in\bbN_0$ and $l\in\{0,\dots,n_\eps-1\}$. By applying (1) and the above computations recursively, we obtain
\begin{align*}
\bbE\left[T_{0,n}^{0}\right]&\leq \bbE\left[T_{0,(k-1)n_\eps+l}^{0}\right]+\bbE\left[T_{(k-1)n_\eps+l,n}^{T_{0,(k-1)n_\eps+l}^{0}}\right]\leq \bbE\left[T_{0,(k-1)n_\eps+l}^{0}\right]+(1+\eps)\bbE\left[T_{0,n_\eps}^{0}\right]\\
&\leq \dots \leq   \bbE\left[T_{0,l}^{0}\right]+(1+\eps)k\,\bbE\left[T_{0,n_\eps}^{0}\right].
\end{align*}
Thus \eqref{303} yields
\[
\frac{\bbE\left[T_{0,n}^{0}\right]}{n}\leq \frac{\bbE\left[T_{0,l}^{0}\right]}{n}+\frac{(1+\eps)kn_\eps(\underline{X}+\eps)}{n}
\]
which then implies 
\[
\limsup_{n\to\infty} \frac{\bbE\left[X_{0,n}^{0}\right]}{n}=\limsup_{n\to\infty} \frac{\bbE\left[T_{0,n}^{0}\right]}{n}\leq (1+\eps)(\underline{X}+\eps).
\]
Since $\eps>0$ was arbitrary, this and \eqref{319} show that 
\beq\lb{305}
\lim_{n\to \infty}\frac{\bbE\left[X_{0,n}^0\right]}{n}=\underline{X}.
\eeq

Next we claim that there is $C_*>0$ such that for any $\eps\in (0,1]$, $n\in\bbN$, and  $i\neq j$ we have
\beq\lb{3claim}
\Var\left[ \xi_i^{n}\right]\leq C_* n^2 \qquad\text{and}\qquad \Cov\left[\xi_i^{n},\xi_j^{n}\right]\leq C_*\eps n^2.
\eeq
We postpone the proof of \eqref{3claim} to the end of the proof of (i).
Since $t_k^{n}=\sum_{i=0}^{k-1}\xi_i^{n}$, we now have
\[
\Var\left[t_k^{n}\right] =\sum_{i=0}^{k-1}\Var\left[ \xi_i^{n}\right] +2\sum_{i=0}^{k-1}\sum_{j=i+1}^{k-1}\Cov\left[\xi_i^{n},\xi_j^{n}\right] \leq   (1+\eps k) C_* n^2k.
\]
Chebyshev's inequality then yields 
\[
\bbP\left[\left|\frac{t_k^{n}-\bbE[t_k^{n}]}{k}\right| > C_\eps \right]\leq  \frac{\Var\left[t_k^{n}\right]}{C_\eps^2 k^2} \leq 
 \frac{1+\eps k}{C_\eps^2 k} C_* n^2.
\]
Since $\bbE[t_k^{n}]=\sum_{i=0}^{k-1}\mu_i^{n}$,  this and \eqref{304} imply
\[
\bbP\left[\frac{t_k^{n}}{k} > (1+\eps)\mu_0^{n}+ C_\eps  \right]\leq \frac{1+\eps k}{C_\eps^2 k} C_* n^2.
\]

For any $N\in\bbN$ write $N=kn+l$, where $k\in\bbN_0$ and $l\in\{0,\cdots,n-1\}$. Then (1) yields
\[
X_{0,N}^0\le T_{0,N}^{0}\leq T_{0,kn}^{0}+T_{kn,kn+l}^{T_{0,kn}^{0}}\leq t_k^{n}+T_{kn,kn+l}^{T_{0,kn}^{0}}.
\]
Denoting $\tau^{n}_N:=T_{kn,kn+l}^{T_{0,kn}^{0}}$, we get $\bbE[\tau^{n}_N]\leq (1+\eps)^2 C'_{\eps} n$ by  \eqref{307} and \eqref{320}, as well as 
\beq\lb{3.0}
\bbP\left[\frac{X_{0,N}^{0}}{N}>(1+\eps)\frac{\mu_0^{n}}{n}+\frac{ C_\eps }{n}+\frac{\tau^{n}_N}{kn}\right]\leq 
\frac{1+\eps k}{C_\eps^2 k} C_* n^2.
\eeq
Now, for each $\eps>0$ pick $n_\eps\in\bbN$ such that
\beq\lb{3.1}
\lim_{\eps\to 0} \frac {C_\eps}{n_\eps}=0 = 
\lim_{\eps\to 0} \frac{\eps n_\eps^2}{C_\eps^2}.
\eeq
If we then take $n=n_\eps$ in \eqref{3.0} and then $N\to\infty$ (so that $k\to\infty$), for each $\delta>0$ we obtain
\[
\limsup_{N\to\infty} \bbP\left[\frac{X_{0,N}^{0}}{N}>(1+\eps)\frac{\mu_0^{n_\eps}}{n_\eps}+\frac{ C_\eps }{n_\eps}+\delta \right]\leq 
\frac{C_* \eps n_\eps^2 }{C_\eps^2}.
\]
  Since $\mu_0^{n_\eps}=\bbE[X_{0,n_\eps}^0]+C_\eps$ and $\lim_{\eps\to 0} n_\eps=\infty$ by \eqref{3.1}, taking $\eps\to 0$ in this estimate and using  \eqref{305} and \eqref{3.1} shows that $\lim_{N\to\infty} \bbP\left[\frac{X_{0,N}^{0}}{N}> \underline{X}+2\delta \right]=0$ for all $\delta>0$.  That is,
\beq\lb{308}
\limsup_{N\to\infty}\frac{X_{0,N}^{0}}{N}
\leq \underline{X} \qquad\text{ in probability}.
\eeq

Let us now assume that there is $\delta>0$ and a sequence $n_k\to\infty$ such that
\beq\lb{361}
    \bbP\left[ \frac{X_{0,n_k}^{0}}{n_k}-\underline{X}<-2\delta\right]\geq 4\delta.
\eeq
Since for all large enough $k$ we have $\frac{\bbE [X_{0,n_k}^0]}{n_k}\ge \underline{X} - \delta$ by  \eqref{305}, we obtain
\begin{align*}
    &\bbP\left[ \frac{X_{0,n}^{0}}{n}-\underline{X}<-2\delta\right]\leq \bbP\left[ \frac{X_{0,n_k}^{0} - \bbE\left[X_{0,n_k}^{0}\right]}{n_k} <-\delta\right]\leq \frac{1}{\delta}\,\bbE\left[\left(\frac{X_{0,n_k}^{0} - \bbE\left[X_{0,n_k}^{0}\right]}{n_k}\right)_+\right]\\
    &\qquad\qquad\qquad\leq  \frac{1}{\delta}\left(\delta^2+\bbP\left[ \frac{X_{0,n_k}^{0} - \bbE\left[X_{0,n_k}^{0}\right]}{n_k} >\delta^2 \right]+\sum_{i\geq 1} \bbP\left[ \frac{X_{0,n_k}^{0} - \bbE\left[X_{0,n_k}^{0}\right]}{n_k} >i\right]\right).
\end{align*}
Since $\Var[X_{0,n_k}^0] \le C_* n_k^2$ by \eqref{3claim} with $i=0$, for $M:=\lceil \frac{C_*}{\delta^2}\rceil+1 $ we obtain
\begin{align*}
\sum_{i\geq M}\bbP\left[ \frac{X_{0,n_k}^{0} - \bbE\left[X_{0,n_k}^{0}\right]}{n_k} >i\right]
\leq \sum_{i\geq M}\frac{1}{n_k^2i^2}\Var\left[X_{0,n_k}^0 \right]\leq \delta^2.
\end{align*}
But \eqref{305} and \eqref{308} also show that for all large enough $k$ we have 
\[
\bbP\left[  \frac{X_{0,n_k}^{0} - \bbE\left[X_{0,n_k}^{0}\right]}{n_k} >\delta^2 \right]
+\sum_{i= 1}^{M-1} \bbP\left[  \frac{X_{0,n_k}^{0} - \bbE\left[X_{0,n_k}^{0}\right]}{n_k} >i\right]\leq \delta^2.
\]
Hence for all large enough $k$ we obtain
\[
\bbP\left[ \frac{X_{0,n_k}^{0}}{n_k}-\underline{X}<-2\delta\right]\leq 3\delta,
\]
which contradicts  \eqref{361}. It follows that $\lim_{n\to\infty}\bbP\left[ \frac{X_{0,n}^{0}}{n}-\underline{X}<-\delta\right]=0$ for each $\delta>0$, so this and \eqref{308} yield \eqref{306}.

It therefore remains to prove \eqref{3claim}.  Similarly as in \eqref{304}, for any $(i,n)\in\bbN_0\times\bbN$ and with $\tilde{\xi}_i^{n}:={\xi}_i^{n}-C_{\eps}$, we get 
\[
\bbE\left[\left(\tilde{\xi}_i^{n}\right)^2\right]\leq (1+\eps)\bbE\left[\left(T_{0,n}^{0}-C_\eps\right)^2\right]=(1+\eps)\bbE\left[\left(X_{0,n}^{0}\right)^2\right],
\]
as well as
\beq\lb{312}
\bbE\left[\left({\xi}_i^{1}\right)^2\right]\leq (1+\eps)\bbE\left[\left(T_{0,1}^{0}\right)^2\right].
\eeq
Since $\Var[\xi_i^n]=\Var[\tilde{\xi}_i^n]$, to prove the first claim in \eqref{3claim}, it suffices to show $\bbE [\left(X_{0,n}^{0}\right)^2]\leq \frac{C_*}2 n^2$ for some $C_*>0$ and all $n\in\bbN$.  We can use $T_{0,n}^{0}\leq  
\sum_{i=0}^{n-1} \xi_i^{1}$ (due to (1)) and  \eqref{312} to obtain
\begin{align*}
\bbE\left[\left(X_{0,n}^{0}\right)^2\right] \leq \bbE\left[\left(T_{0,n}^{0}\right)^2\right]
\leq n \bbE \left[\sum_{i=0}^{n-1} (\xi_i^{1})^2\right]
\leq (1+\eps) n^2 \bbE\left[\left(T_{0,1}^{0}\right)^2\right],
\end{align*}
which yields this estimate with $C_*:= 4 \bbE\left[\left(X_{0,1}^{0}+C_{1}\right)^2\right]$.  

To prove the second claim in \eqref{3claim}, 
we apply ($4^{*}$) to get that for any $i,i',j,j',k,l\in \bbN_0$ satisfying $l>k$ that
\beq  \lb{9.1}
\begin{aligned}
\{t_k^{n} =i'\}\in \calF^-_{i'-C_{\eps}}, &\qquad \{T_{kn,(k+1)n}^{i'} =i\}\in  \calF^-_{i'+i-C_{\eps}} \cap \calF^+_{i'}, \\
\{t_l^{n} =j'\}\in \calF^-_{j'-C_{\eps}}, &\qquad  \{T_{ln,(l+1)n}^{j'} =j\}\in \calF^+_{j'}.
\end{aligned}
\eeq
Note that $l>k$ implies $j'\geq i+i'$  whenever $ \bbP [T_{kn,(k+1)n}^{i'}=i \,\, \& \,\, t_k^{n}=i'\,\, \& \,\, t_l^{n}=j']>0$  because
\[
t_l^{n}=t_k^{n}+\xi_{k}^{n}+\cdots+\xi_{l-1}^{n}\geq t_k^{n}+\xi_{k}^{n}=t_k^{n}+T_{kn,(k+1)n}^{t_k^{n}}.
\]
Recalling that $\tilde \xi_k^{n}=T_{kn,(k+1)n}^{t_k^{n}}-C_\eps\ge 0$,
it follows from the above, (3), \eqref{301}, and \eqref{309} that
\begin{align*}
    \bbE & \left[\tilde{\xi}_k^{n}\tilde{\xi}_l^{n}\right]=\sum_{i,j,i',j'\geq 0}(i-C_{\eps})(j-C_{\eps})\,   
    \bbP\left[ T_{ln,(l+1)n}^{j'}=j \,\, \& \,\, t_l^{n}=j'  \,\, \& \,\,  T_{kn,(k+1)n}^{i'}=i \,\, \& \,\, t_k^{n}=i' \right]\\
    & \leq (1+\eps) \sum_{i,j,i',j'\geq 0}(i-C_{\eps})(j-C_{\eps})\,
    \bbP\left[T_{ln,(l+1)n}^{j'}=j\right] \bbP\left[ t_l^{n}=j' \,\, \& \,\,  T_{kn,(k+1)n}^{i'}=i \,\, \& \,\, t_k^{n}=i' \right] \\
    & = (1+\eps) \sum_{i,j,i'\ge 0}(i-C_{\eps})(j-C_{\eps})\,
    \bbP\left[T_{0,n}^{0}=j\right]  \bbP\left[T_{kn,(k+1)n}^{i'}=i  \,\, \& \,\, t_k^{n}=i'\right] \\
   &  \leq (1+\eps) \bbE\left[\tilde{\xi}^{n}_0 \right]\bbE\left[\tilde{\xi}^{n}_k\right]
    \leq (1+\eps)^2 
  \bbE\left[\tilde{\xi}^{n}_0\right] ^2,
\end{align*}
where we used that the summands are zero whenever $j'<i'+i$.
Also note that \eqref{309'} yields 
\[
\bbE\left[\tilde{\xi}^{n}_k\right]
    \bbE\left[\tilde{\xi}^{n}_l\right]\geq (1-\eps)^2    \bbE\left[\tilde{\xi}^{n}_0\right] ^2,
\]
hence
\beq\lb{321}
\Cov\left[\xi_k^{n},\xi_l^{n}\right]=\Cov\left[\tilde{\xi}_k^{n},\tilde{\xi}_l^{n}\right]\leq 
4\eps   \bbE\left[\tilde{\xi}^{n}_0\right] ^2 = 4\eps   \bbE\left[ X_{0,n}^0 \right] ^2.
\eeq
Now the second claim in \eqref{3claim} follows by \eqref{305}, and the proof of (i) is finished.
%
\end{proof}

Next we adjust this proof to obtain the integer-valued version of the first claim in Theorem~\ref{T.1.2}.  We will use in it the following lemma.

\begin{lemma} \lb{L.3.0}
For $E,F\in\calF$, let $\Psi(E,F):=\max\{s\in\bbZ \,|\,(E,F)\in\calF_{t}^- \times \calF_{t+s}^+ \text{ for some $t\in\bbN_0$}\}$ (if there is no such $s$, then $\Phi(E,F):=-\infty$).  Assume that $A_j^k,B_k\in\calF$  ($k,j\in\bbN_0$) are such that $B_0,B_1,\dots$ are pairwise disjoint, and so are $A_0^k,A_1^k,\dots$ for each $k\in\bbN_0$.  If $s:=\min\{\Psi(B_k,A_j^k)\,|\, j,k\in\bbN_0\}\ge 0$ and  $S:=\sup\{f(j,k)\,|\,\bbP[A_j^k\cap B_k]>0\}$ for some $f:\bbN_0^2\to[0,\infty)$, then with $\phi$ from Remark 2 after Theorem \ref{T.1.2} we have
\[
 \sum_{j,k\ge 0}  f(j,k)   \left|\bbP[A_j^k\cap B_k] - \bbP[A_j^k]\bbP[B_k] \right| \le 2S\phi(s).
\]
\end{lemma}

\begin{proof}
Let 
\[
U^\pm:= \left\{(j,k)\in\bbN_0^2\,\big|\,  \pm (\bbP[A_j^k\cap B_k] - \bbP[A_j^k]\bbP[B_k])>0 \right\},
\]
and let $U_k^\pm:=\{j\in\bbN_0\,|\,(j,k)\in U^+\}$ for each $k\in\bbN_0$.
Then
\begin{align*}
 \sum_{(j,k)\in U^+}  f(j,k)   \left|\bbP[A_j^k\cap B_k] - \bbP[A_j^k]\bbP[B_k] \right| \le  S \sum_{k\ge 0}     \left(\bbP \left[ \bigcup_{j\in U_k^+} A_j^k\cap B_k \right] - \bbP \left[ \bigcup_{j\in U_k^+} A_j^k \right]\bbP[B_k] \right),
\end{align*}
which  $\le S\phi(s)$ because if $t_k\in\bbN_0$ is minimal such that $B_k\in \calF_{t_k}^-$, then  $\bigcup_{j\in U_k^+} A_j^k\in\calF_{t_k+s}^+$.  The same estimate holds for the sum over $U^-$, finishing the proof.
\end{proof}

\begin{theorem} \lb{T.2.2}
Assume the hypotheses of Theorem \ref{T.2.1}, but with
($2^{**}$) and ($5^{**}$) replaced by ($2^{*}$) and $\lim_{s\to\infty}\phi(s)=0$ for $\phi$ from Remark 2 after Theorem \ref{T.1.2} (with $s,t_0,t_1,\dots\in\bbN_0$).
%
%
Then \eqref{306} holds.
\end{theorem}

\begin{proof}
%
%
This proof follows along the same lines as the one of Theorem \ref{T.2.1}, with some minor adjustments.
From (1), ($2^{*}$), and (3) we see that for any integers $t\geq 0$ and $n>m\geq0$ we have
\beq\lb{2**}
X_{m,n}^t\le C(n-m).
\eeq
With the $\phi$ considered here, let $C_\eps\in\bbN$ be such that
\beq\lb{301'}
\phi( C_{\eps} )\le \frac\eps 2,
\eeq
and let $T_{m,n}^t, \underline{X}, t_i^n,\xi_i^n,\mu_i^n$ be defined as before.
Then \eqref{2**}, Lemma \ref{L.3.0}, and (3) yield
\beq \lb{2.4'}
\begin{aligned}
   \mu_i^{n} &=\sum_{k,j\geq 0}
   j\,\bbP\left[T_{in,(i+1)n}^{k}=j \,\, \& \,\, t_i^{n}=k\right] \\
   &\leq \sum_{k, j\geq 0} j\,\bbP\left[T_{in,(i+1)n}^{k}=j \right]\bbP\left[ t_i^{n}=k\right]+(Cn+C_\eps)\eps  \\
   & = \sum_{k,j\geq 0} j\,\bbP\left[T_{0,n}^{0}=j\right]\bbP\left[ t_i^{n}=k\right] +(Cn+C_\eps)\eps\\
    &= \mu_0^{n}+(Cn+C_\eps)\eps
\end{aligned}
\eeq
instead of \eqref{304}.
Similarly, we obtain
\beq\lb{2.5'}
\bbE\left[\xi_i^n-C_\eps\right]\leq \mu_0^{n}-C_\eps+Cn\eps.
\eeq
and
\beq\lb{2.6'}
\bbE\left[\xi_i^n-C_\eps\right]\geq \mu_0^{n}-C_\eps-Cn\eps.
\eeq
Using (1) and \eqref{2.4'} in place of \eqref{304}, we now get
\beq\lb{2.7'}
\mu_0^n\leq C_\eps'n
\eeq
in place of \eqref{320}, with $C_\eps':=\bbE\left[X_{0,1}^0\right]+C_\eps+C\eps$.

Next, similarly to \eqref{2.4'} and using Lemma \ref{L.3.0} and \eqref{2**}, we can replace \eqref{307} by
\beq\lb{2.9'}
\begin{aligned}
\bbE\left[T_{m,n}^{T_{0,m}^{0}}\right] &=\sum_{i,j\geq 0}
j\,\bbP\left[T_{m,n}^{i}=j \,\, \& \,\,  T_{0,m}^{0}=i\right]\\
&\leq \sum_{i,j\geq 0} j\,\bbP\left[T_{m,n}^{i}=j\right]\bbP\left[ T_{0,m}^{0}=i\right]
+ (C(n-m)+C_\eps)\eps \\
&= \bbE\left[ T_{0,n-m}^{0}\right]+ (C(n-m)+C_\eps)\eps .
\end{aligned}
\eeq
With this, we again obtain \eqref{305}.

The proof of \eqref{3claim} is also adjusted similarly to \eqref{2.4'}.  We now obtain
\[
\bbE\left[\left(\tilde{\xi}_i^{n}\right)^2\right]\leq \bbE\left[\left(T_{0,n}^{0}-C_\eps\right)^2\right]+(Cn)^2\eps=\bbE\left[\left(X_{0,n}^{0}\right)^2\right]+(Cn)^2\eps
\]
and
\[
\bbE\left[(\xi_i^1)^2\right]\leq \bbE\left[(T_{0,1}^0)^2\right]+(C+C_\eps)^2\eps,
\]
which yields the first claim in \eqref{3claim} as before (with a different $C_*$).  In the proof of the second claim, we use \eqref{2.5'} in place of \eqref{309}, as well as $\tilde{\xi}_k^n\leq Cn$ (due to \eqref{2**}).  We also use the same adjustment as in \eqref{2.4'}, but now replacing the sum over $k$ by the sum over $(i,i',j')$ (with $A_j^{(i,i',j')}:=\{T_{ln,(l+1)n}^{j'}=j\}$ when we use Lemma \ref{L.3.0}).  This and \eqref{2**} show that 
\begin{align*}
    \bbE & \left[\tilde{\xi}_k^{n}\tilde{\xi}_l^{n}\right]=\sum_{i,j,i',j'\geq 0}(i-C_{\eps})(j-C_{\eps})\,   
    \bbP\left[ T_{ln,(l+1)n}^{j'}=j \,\, \& \,\, t_l^{n}=j' \,\, \& \,\, T_{kn,(k+1)n}^{i'}=i \,\, \& \,\, t_k^{n}=i'  \right]\\
    & \leq  \sum_{i,j,i',j'\geq 0}(i-C_{\eps})(j-C_{\eps})\,
    \bbP\left[T_{ln,(l+1)n}^{j'}=j\right] \bbP\left[ t_l^{n}=j' \,\, \& \,\, T_{kn,(k+1)n}^{i'}=i \,\, \& \,\, t_k^{n}=i' \right] 
    +(Cn)^2\eps\\
    & =  \sum_{i,j,i'\ge 0}(i-C_{\eps})(j-C_{\eps})\,
    \bbP\left[T_{0,n}^{0}=j\right]  \bbP\left[T_{kn,(k+1)n}^{i'}=i  \,\, \& \,\, t_k^{n}=i'\right] 
	+(Cn)^2\eps\\
    &  =  \bbE\left[\tilde{\xi}^{n}_0\right] \bbE\left[\tilde{\xi}^{n}_k\right] +(Cn)^2\eps 
 \leq  \bbE\left[\tilde{\xi}^{n}_0\right]^2 + Cn\eps \bbE\left[\tilde{\xi}^{n}_0\right] +(Cn)^2\eps.
\end{align*}
This, \eqref{2.6'} applied with $i=k,l$, and $\tilde{\xi}^{n}_0\le Cn$ then yield the second claim in \eqref{3claim} with $C_*:=4C^2$.

Now, the proof of \eqref{3.0}, but with  \eqref{304}, \eqref{320}, and \eqref{307} replaced by \eqref{2.4'}, \eqref{2.7'}, and \eqref{2.9'}, shows that
\[
\bbP\left[\frac{X_{0,N}^{0}}{N}>\frac{\mu_0^{n}}{n}+C\eps+\frac{(1+\eps) C_\eps }{n}+\frac{\tau^{n}_N}{kn}\right]\leq 
\frac{1+\eps k}{C_\eps^2 k} C_* n^2,
\]
where $\tau^{n}_N:=T_{kn,kn+l}^{T_{0,kn}^{0}}$ satisfies $\bbE[\tau^{n}_N]\leq (C'_{\eps} +(C+C_\eps)\eps)n$.
This then implies \eqref{308} as before, and the rest of the proof is identical to the proof of Theorem \ref{T.2.1}. 
\end{proof}

We can now prove  the first claim in Theorem \ref{T.1.2} similarly to the proof of Theorem \ref{T.1.1}.

\begin{proof}[Proof of the first claim in Theorem \ref{T.1.2}]
Let us first assume that $c\geq 1$.  Let
\[
T_{m,n}^t:= \lceil X_{m,n}^t +C \rceil\quad (\in\bbN_0)
\]
and restrict $t$ to $\bbN_0$.  
Similarly to the proof of Theorem \ref{T.1.1}, we find that $T_{m,n}^t$ satisfies hypotheses (1), (3), ($4^{**}$), ($6^{**}$) of Theorem~\ref{T.2.2} (with $X_{m,n}^t$ replaced by $T_{m,n}^t$), but with $\max\{\lceil C\rceil,1\} $ in place of $C$ in ($6^{**}$).  Hence iteration of ($6^{**}$) shows that it also holds for $T_{m,n}^t$ and $C':=2 \max\{\lceil C\rceil,1\}$ in place of $C$.  From ($2^{*}$) for $X_{m,n}^t$ we see that $T_{m,n}^t$ also satisfies ($2^{*}$) with $C'$ in place of $C$.

Let now $\phi$ be as in Remark 2 after Theorem \ref{T.1.2}.  Note that if we define $\tilde \phi(s)$ as in that remark but only with $s,t_0,t_1,\dots\in\bbN_0$, then $\tilde \phi\le \phi$.  Therefore our hypothesis \hbox{$\lim_{s\to\infty}\phi(s)=0$} implies the last hypothesis in Theorem~\ref{T.2.2} as well.  That theorem for $T_{m,n}^t$ now  yields \eqref{306}.

For $c\in(0,1)$, we let $\calG_t^\pm$ and $Y_{m,n}^t$ be as in the proof of Theorem \ref{T.1.1}.  The above argument with $(\calG_t^\pm,Y_{m,n}^t,SC, Sc)$ in place of $(\calF_t^\pm,X_{m,n}^t,C,c)$ then again concludes \eqref{306}. 

Finally, in the setting of the last claim in Theorem \ref{T.1.2} we can just apply Theorem \ref{T.2.2} directly to $X_{m,n}^t$ 
(with $\tilde \phi$ above).
\end{proof}


\section{Time-Decaying Dependence II}  \lb{S3'}

In this section we will prove the second claim in Theorem \ref{T.1.2}, as well as the corresponding integer-valued claim.


\begin{proof}[Proof of the second claim in Theorem \ref{T.1.2}]
Similarly to the proof of the first claim in Theorem \ref{T.1.2}, this again follows from the corresponding integer-valued claim.  Hence, without loss, we can restrict $t$ to $\bbN_0$ and assume that $X_{m,n}^t$ only takes values in $\bbN_0$.

The first claim in Theorem \ref{T.1.2} yields 
\beq\lb{6.2}
\lim_{n\to\infty}\frac{X_{0,n}^0}{n} =\lim_{n\to \infty}\frac{\bbE\left[X_{0,n}^{0}\right]}{n}=:\underline{X}\ge 0\qquad\text{ in probability}.
\eeq
Let us now prove
\beq\lb{9.8}
 \limsup_{n\to\infty}\frac{X_{0,n}^0}{n} \leq\underline{X} \qquad\text{ almost surely.}
\eeq

As in the proof of Theorem \ref{T.2.1}, let $T_{m,n}^t:=X_{m,n}^t+C_\eps$ some $C_\eps\in\bbN$ that is a multiple of $C$ and \eqref{301'} also holds.  Then (1'), (3'), (6') from Theorem \ref{T.2.5} hold and so does (4'') from the proof of Theorem \ref{T.2.1}, while (2') is replaced by $T_{0,1}^0\leq C+C_\eps$, and ($5^*$) also holds.

For any $n\in\bbN$,  define  $t_i^n$ and $\xi_i^n$ as at the start of the proof of Theorem \ref{T.2.1}.
From ($4$'') we again get \eqref{9.1} for any $i,i',j,j',k,l\in \bbN_0$, and the argument after \eqref{9.1} again shows that if
$l>k$, then $j'\geq i+i'$ whenever $ \bbP [T_{kn,(k+1)n}^{i'}=i \,\, \& \,\, t_k^{n}=i'\,\, \& \,\, t_l^{n}=j']>0$.
Then the argument from the proof of the second claim in \eqref{3claim} in the proof of Theorem \ref{T.2.2} (which uses Lemma~\ref{L.3.0}) shows
that for any $\nu,\nu'\in\bbN$ we have
\begin{align}
    \bbP & \left[{\xi}_k^{n}\ge \nu \,\,\&\,\,{\xi}_l^{n}\ge\nu'\right]=\sum_{i-\nu,j-\nu',i',j'\geq 0}
    \bbP\left[ T_{ln,(l+1)n}^{j'}=j \,\, \& \,\, t_l^{n}=j'  \,\, \& \,\,  T_{kn,(k+1)n}^{i'}=i \,\, \& \,\, t_k^{n}=i' \right] \notag \\
    & \leq  \sum_{i-\nu,j-\nu',i',j'\geq 0}\,
    \bbP\left[T_{ln,(l+1)n}^{j'}=j\right] \bbP\left[ t_l^{n}=j' \,\, \& \,\,  T_{kn,(k+1)n}^{i'}=i \,\, \& \,\, t_k^{n}=i' \right] +\eps 
    \notag \\
    & =  \sum_{i-\nu,j-\nu',i'\ge 0}\,
    \bbP\left[T_{0,n}^{0}=j\right]  \bbP\left[T_{kn,(k+1)n}^{i'}=i  \,\, \& \,\, t_k^{n}=i'\right] +\eps  \lb{6.8}  \\
   &  \leq   \bbP\left[T_{0,n}^{0}\geq \nu'\right] \sum_{i-\nu,i'\ge 0} \bbP\left[T_{kn,(k+1)n}^{i'}=i \right]\bbP\left[ t_k^{n}=i'\right]+2\eps \notag \\
   &=  \bbP\left[T_{0,n}^{0}\geq \nu'\right] \bbP[T_{0,n}^{0}\geq \nu]+2\eps. \notag 
\end{align}


Now  fix some $K\in\bbN$.
From (1') we see that for any $n\in\bbN$ we  have
\beq\lb{6.5}
 T_{0,Kn}^0 \leq \sum_{i=0}^{K-1} T_{in,(i+1)n}^{t^{n}_{i} }.
\eeq
From \eqref{6.2} we see that there is $\eps$-independent $n_{K}\in\bbN$  such that for all $n\geq \max\{C_\eps K,n_{K}\}$,
\beq\lb{6.7}
\bbP\left[\frac{T_{0,n}^0}{n}-\underline{X}\geq \frac{2}{K}\right]\leq \bbP\left[\frac{X_{0,n}^0}{n}-\underline{X}\geq \frac{1}{K}\right]\leq \frac{1}{2K^2}.
\eeq
From (1), ($2^*$), and (3) we get $T_{in,(i+1)n}^{t^{n}_{i}} \leq C_\eps+Cn\leq (C+1)n$ for these $n$ and all $i\in\bbN_0$. 
This means that if only one of the numbers
\[
g_i^n:=\frac{T_{in,(i+1)n}^{t^{n}_{i} }}n -\underline{X}- \frac 2K \qquad (i\in \{0,\ldots,K-1\})
\]
is positive, then \eqref{6.5} yields
\[
\frac{T_{0,Kn}^0}{Kn}-\underline{X}\le \sum_{i=0}^{K-1} \left( \frac{g_i^n}K+\frac 2{K^2} \right) < \frac{C+1}K + \frac 2K =  \frac{C+3}{K}.
\]
The same estimate holds if each of these numbers is less than $\frac{C+1}{K}$. These facts, \eqref{6.8}, (3'), and \eqref{6.7} now imply that for any $n\geq \max\{C_\eps K,n_{K}\}$ we have
\begin{align*}
\bbP\left[\frac{T_{0,Kn}^0}{Kn}-\underline{X}\geq \frac{C+3}{K}\right]
& \leq \sum_{0\leq i,j <K \,\&\,i\neq j}\bbP\left[g_i^n\geq \frac{C+1}{K} \,\,\&\,\, g_j^n\geq 0\right]\\
& \leq K^2 \bbP\left[g_0^n\geq \frac{C+1}{K}\right]\bbP\left[g_0^n\geq 0\right]+2K^2\eps 
\\ &\leq \frac12 \bbP\left[\frac{T^0_{0,n}}{n}-\underline{X}\geq \frac{C+3}{K}\right]+2K^2\eps.
\end{align*}

We can now apply this estimate iteratively with $Kn, K^2n, \dots$ in place of $n$ and obtain for any $n\geq \max\{C_\eps K,n_{K}\}$  and $q\in\bbN$,
 \begin{align*}
\bbP\left[\frac{T_{0,K^q n}^0}{K^qn}-\underline{X}\geq \frac{C+3}{K}\right]\leq 2^{-q}\,\bbP\left[\frac{T^0_{0,n}}{n}-\underline{X}\geq \frac{C+3}{K}\right]+4K^2\eps\leq 2^{-q}+4K^2\eps.
 \end{align*}
 This of course also yields
 \beq\lb{9.2}
\bbP\left[\frac{X_{0,K^q n}^0}{K^qn}-\underline{X}\geq \frac{C+3}{K}\right]\leq 2^{-q}+4K^2\eps.
\eeq
%
%
The hypothesis shows that there is $A\in\bbN$ such that $C_\eps \leq A\eps^{-A}$ for all $\eps\in(0,1)$.  Let
\[
 C'\ge C_K:=AKn_K\qquad\text{and}\qquad  M_K:=2^AK,
\]
with $C'\in\bbN$.  Then for any  $q\in\bbN$, \eqref{9.2} with $\eps:=2^{-q}$  and $n:=2^{Aq} C'$ ($\ge \max\{C_\eps K,n_{K}\}$) yields 
\beq
\bbP\left[\frac{X_{0,C' M_K^q }^0}{C' M_K^q }-\underline{X}\geq \frac{C+3}{K}\right]\leq 5K^2 2^{-q}.
\eeq 
By the Borel-Cantelli Lemma we then obtain
 \beq\lb{6.10}
 \limsup_{q\to \infty} \frac{X_{0,C' M_K^q }^0}{C' M_K^q}\leq \underline{X}+\frac{C+3}K\qquad\text{ almost surely}.
 \eeq
 
 Now 
 apply \eqref{6.10} with $C'$ taking all the values in 
 \[
 U_K:= \left\{C_K,C_K+1,\dots,C_KM_K \right\}.
 \]
Then for any large $n$, there is $(C',q)\in U_K \times \bbN$ such that $C'M_K^q\le n\le C'M_K^q+C_K^{-1}n$ 
 and
 \[
 \frac{X_{0,C' M_K^q }^0}{C' M_K^q}\leq \underline{X}+\frac{C+4}K.
 \]
So by (1) and ($2^*$) we have
$X_{0,n}^0\leq X_{0,C' M_K^q }^0+CC_K^{-1} n$,
which yields
\[
 \limsup_{n\to \infty} \frac{X_{0,n }^0}{n}\leq \underline{X}+\frac{C+4}K+ CC_K^{-1} \qquad\text{ almost surely}.
\]
By taking 
 $K\to\infty$, we conclude  \eqref{9.8}.

It remains to prove 
 \beq \lb{9.3}
 \underline{X} \le \liminf_{n\to\infty}\frac{X_{0,n}^0}{n}  \qquad\text{ almost surely.}
\eeq
We will do this with only assuming $\lim_{s\to\infty}  \phi(s)=0$ (rather than $\lim_{s\to\infty} s^\alpha \phi(s)=0$), and without the use of the proof of \eqref{9.8}.  This will then also prove Remark 3 after Theorem \ref{T.1.2}.

For any $t,m,n,j\in\bbN_0$ with $j\ge n$, let
\[
Y_{m;n,j}^t:=\min_{n\le i \le j } \frac{X_{m,m+i}^t}{i}\qquad\text{and}\qquad  Z_m^t:=\lim_{n\to\infty}\lim_{j\to\infty}Y_{m;n,j}^t=\liminf_{n\to\infty} \frac{X_{m,m+n}^t} n.
\]
$Z_m^{t+Ck}$ is non-decreasing in $k\in\bbN$ by (6) (with $c=0$), and since the law of $Z_m^t$ is independent of $(t,m)$ by (3), we almost surely have $Z_m^{t+Ck}=Z_m^t$ for all $k\in\bbN$. Moreover we claim that  $Z_0^0$ is almost everywhere  constant (which  implies that $Z_m^t$ is a.e.~equal to the same constant for each $(t,m)\in\bbN_0$).

If this is not the case, let $c:=\Var[Z_0^0]>0$. 
From (1), ($2^*$), and (3) we have
\beq\lb{6.11}
\max \left\{ Z_0^0,\,Y_{0;n,j}^t,\,\frac{X_{0,n}^t}{n} \right\} \leq C\qquad\text{ for all $t,n,j\in\bbN_0$ with $j\ge n\geq1$}.
\eeq
Let  $\delta:=\frac{c}{4C(C+1)}$. By (3) and Ergorov's Theorem, there $\delta$-dependent  $n,j\in\bbN^2$ with $j\ge n$  such that for any $t\in\bbN_0$ we have
\[
|Y_{0;n,j}^t-Z_0^t|\leq \delta\qquad\text{ on some }\Omega^t_{\delta}\subseteq\Omega \text{ with }\bbP[\Omega^t_{\delta}]\geq 1-\delta.
\]
Also since $Z_0^0=Z_0^{Ck}$ a.e. for all $k\in\bbN$, \eqref{6.11} and $\Var[Z_0^0]= c$ imply
\beq\lb{6.12}
\Cov\left[Y_{0;n,j}^0,Y_{0;n,j}^{Ck}\right]\geq \Cov\left[Z_0^0,Y_{0;n,j}^{Ck}\right]-(C+C^2)\delta\geq \Cov\left[Z_0^0,Z_0^{Ck}\right]-2(C+C^2)\delta\geq \frac{c}2.
\eeq

Next note that by (4), \eqref{6.11}, and the definition of $Y_{0;n,j}^0$ we have
\[
Y^{0}_{0;n,j}\text{ is }\calF^-_{Cj}\text{-measurable\qquad and\qquad }Y^{Ck}_{0;n,j}\text{ is } \calF^+_{Ck}\text{-measurable}. 
\]
This, Lemma \ref{L.3.0}, \eqref{6.11}, and $Y_{0;n,j}^t$  only taking rational values
show for any $k\ge j\ge n\ge 1$,
\beq\lb{6.18}
\begin{aligned}
  \bbE\left[Y_{0;n,j}^0Y_{0;n,j}^{Ck}\right] & =\sum_{p,q\in\bbQ} pq\,   
    \bbP\left[ Y_{0;n,j}^{Ck}=p \,\, \& \,\,  Y_{0;n,j}^{0}=q \right]\\ &
    \leq \sum_{p,q\in\bbQ} pq\,   
    \bbP\left[ Y_{0;n,j}^{Ck}=p\right]\bbP\left[ Y_{0;n,j}^{0}=q \right]+C^2\phi(C(k-j))\\
    & = \bbE\left[Y_{0;n,j}^0\right]\bbE\left[ Y_{0;n,j}^{Ck}\right]+C^2\phi(C(k-j)),
\end{aligned}
\eeq
Hence $\Cov\left[Y_{0;n,j}^0,Y_{0;n,j}^{Ck}\right]\leq C^2\phi(C(k-j))$, which contradicts with \eqref{6.12} if we take $k$ large enough (because   ($5^{*}$) holds). 

Therefore $Z_0^0$ is indeed almost everywhere equal to some constant $Q\in[0,\underline{X}]$.
Then \eqref{9.3} is just $\underline{X}\leq Q$, so we only need to prove this. 
For any $\eps>0$ and $K\in\bbN$, let us  define 
\beq\lb{9.5}
T_{m,n}^t:=X_{Km,Kn}^t+C_\eps
\eeq
(which depends on $\eps,K$ but we suppress this in the notation).
Then again  (1'), (3'), (6') from Theorem \ref{T.2.5} hold and so does (4'') from the proof of Theorem \ref{T.2.1}, while (2') is replaced by $T_{0,1}^0\leq CK+C_\eps$, and ($5^*$) also holds.
From (1'), (3'), and $T_{0,1}^0\leq CK+C_\eps$ we obtain for any $t,m,n\in\bbN_0$,
\beq\lb{6.19}
T_{m,m+n}^t\leq CKn+C_\eps.
\eeq
Note that to prove $\underline{X}\leq Q$, it suffices to show that
\beq\lb{6.21}
\frac{\bbE[T_{0,n}^0]}{n}\leq KQ+KC\eps+ \eps(C_\eps+2)+\frac{M_{K,\eps}'}{n}
\eeq
holds for each $\eps>0$ and $K\in\bbN$, with some $n$-independent $M_{K,\eps}'$. 
This is because after dividing  \eqref{6.21} by $K$ and taking $n\to\infty$, 
we obtain from  \eqref{6.2},
\[
\underline{X} = \lim_{n\to\infty} \frac{\bbE[X_{0,Kn}^0]}{Kn} \le Q+C\eps+ \frac{\eps(C_\eps+2)}K.
\]
Taking $K\to\infty$ and then $\eps\to 0$ now  yields $\underline{X}\le Q$, so we are indeed left with proving \eqref{6.21}.

This is done similarly to the argument in the proof of
\eqref{110},  with $KQ$  in place of $Q$. Fix $\eps>0$ and $K\in\bbN$, let $Q_\eps:=KQ+\eps$ (as at the start of that proof), and let $T_{m,n}^t$ be from \eqref{9.5}.  Note that for any $t,m\in\bbN_0$ we have $\liminf_{n\to\infty}\frac{T^t_{m,m+n}}{n}=KQ$ almost surely because  $Z_m^t=Q$ almost everywhere.   Define 
\[
N_m^t, M_\eps,t_k,r_k,S_n
\]
as in the proof of  \eqref{110}, and follow that proof, with two adjustments  near the end where  (5') was used. The first is the estimate on
\[
\bbE \left[ T_{r_{k} ,r_{k} +1}^{t_{k} } 1_{\{N^{t_{k} }_{r_{k} } >M_\eps\}}\right].
\]
From (4'') we have for any $i,j\in\bbN_0$ that $\{r_{k}=i \,\,\&\,\, t_k=j\}\in \calF^{-}_{j-C_\eps}$, and  $T_{i,i+1}^j$ and $N^{j }_{i}$ are $\calF_j^+$-measurable. Hence we can use ($5^{*}$), \eqref{6.19}, and Lemma \ref{L.3.0} instead of (5')  (as well as (3') and \eqref{112} as before) to obtain
\beq\lb{6.15}
\begin{aligned}
   \bbE &\left[ T_{r_{k} ,r_{k} +1}^{t_{k} } 1_{\left\{N^{t_{k} }_{r_{k} } >M_\eps\right\}}\right]
   =\sum_{i,j \ge 0}  \bbE\left[T_{i,i+1}^{j} 1_{\left\{N^{j }_{i} >M_\eps\right\}} 1_{\left\{r_{k}=i \,\, \& \,\, t_{k}=j \right\} }\right]
   \\
&\leq\,\sum_{i,j,l \ge 0} l\,\bbP\left[ T_{i,i+1}^{j}=l\,\,\&\,\, N^{j}_{i}>M_\eps\right]\bbP\left[ r_{k} =i \,\, \& \,\, t_{k} =j\right]+(CK+C_\eps)\eps\\
&=\,\bbE\left[ T_{0,1}^0 1_{\left\{N^{0}_{0}>M_\eps\right\}}\right]+(CK+C_\eps)\eps\leq (CK+C_\eps+1)\eps.
\end{aligned}
\eeq
This then yields
\beq\lb{6.20}
\frac{\bbE[S_n]}{n}\leq KQ+(CK+C_\eps+2)\eps
\eeq
in place of \eqref{113}. The second place 
is the estimate on $\bbE [ T_{n-l ,n}^{t_{K_n }} ]$, but here we can simply use \eqref{6.19} to obtain
\[
\bbE [ T_{n-l ,n}^{t_{K_n }} ] \le CKl+C_\eps.
\]
This and \eqref{300} yield
\beq\lb{6.22}
\bbE[T^{t_{Kn}}_{r_{K_n},n}]\leq \sum_{l=1}^{M_\eps}\bbE \left[ T_{n-l ,n}^{t_{K_n }} \right]\leq M_\eps(CKM_\eps+C_\eps)=:M_{K,\eps}'.
\eeq
This, \eqref{6.20}, and \eqref{114} now show \eqref{6.21}, and the proof is finished.
\end{proof}

\section{PDE and First Passage Percolation in Time-Dependent Environments} \lb{S4}

Our main motivation for this work was its application in the proofs of homogenization for reaction-diffusion equations and $G$-equations with time-dependent coefficients \cite{ZhaZla5}.  However, our results  can be used to study propagation of solutions to even more general PDE.

\begin{example} \lb{E.4.3}
Consider some PDE on $[0,\infty)\times \bbR^d$ with space-time stationary coefficients, for which the maximum principle holds.  Assume that  (5) resp.~($5^*$) holds when $\calF_t^\pm$ are  $\sigma$-algebras generated by the coefficients restricted to  $[0,t]\times \bbR^d$ and $[t,\infty)\times \bbR^d$, respectively.  Fix some compactly supported ``bump'' function $u_0:\bbR^d\to[0,\infty)$, and for any $(t',x')\in\bbR^d$ let $u^{t',x'}$ solve the PDE with initial value $u^{t',x'}(t',\cdot):=u_0(\cdot-x')$.  Then for any $y\in\bbR^d$ let
\[
X^{t'}(x',y):=\inf\left\{ t\ge 0 \,\Big|\, u^{t',x'}(t+t',\cdot)\ge u_0(\cdot -y)  \right\},
\]
so that $X^{t'}(x',y)$ can be thought of as the time it takes for $u^{t',x'}$ to propagate from $x'$ to $y$, starting at time $t'$.  Let us also assume that $u_0$ was chosen so that for some $C\ge 0$ and all $t'\ge C$ we have
$u^{0,0}(t',\cdot)\ge u_0$.

Fix any $t\in[0,\infty)$ and  unit vector $e\in\bbS^{d-1}$, and let  $X_{m,n}^{t,e}:=X^{t}(me,ne)$.  Then (4) is obvious from the definition of $X_{m,n}^{t,e}$, while maximum principle, space-time stationarity of coefficients, and $u^{0,0}(t',\cdot)\ge u_0$ for all $t'\ge C$ yield (1), (3), and (6).  
Hence if (2) resp.~($2^*$) holds, Theorem \ref{T.1.1} resp.~\ref{T.1.2} can be used to show that  the  limit
\beq\lb{5.4}
\lim_{n\to\infty} \frac{X_{0,n}^{0,e} }n
\eeq
exists and equals a constant (almost surely or in probability).  Of course, its reciprocal then represents the deterministic asymptotic speed of propagation in direction $e$ for this PDE.

In fact, if $\frac{X^{t'}(x',y)}{|x'-y|}$ is bounded below and above by positive constants $c_0\le c_1$ whenever $|x'-y|\ge 1$, then (2) and ($2^*$) clearly hold, asymptotic propagation speeds in all directions are between $\frac1{c_1}$ and $\frac1{c_0}$, and the PDE even has a deterministic asymptotic shape of propagation (called {\it Wulff shape}).  Indeed, a version of a standard argument going back to \cite{Ric, CoxDur} (see \cite{ZhaZla5}) can typically be used to show that there is a convex open set $S\subseteq\bbR^d$, containing  and contained in the balls centered at the origin with radii $\frac1{c_1}$ and $\frac1{c_0}$, respectively, such that if $S_t(\omega):=\{ x \in\bbR^d \,|\, X^0(0,x)\le t\}$, then for any $\delta>0$ we have
\[
(1-\delta)tS \subseteq S_t(\omega) \subseteq (1+\delta)tS,
\]
either for almost every $\omega\in\Omega$ and all large-enough $t\ge 0$ (depending on $\omega$ and $\delta$) or with probability converging to 1 as $t\to\infty$.

We refer the reader to our companion paper \cite{ZhaZla5} for further details and specific applications of Theorems \ref{T.1.1} and \ref{T.1.2} to homogenization for reaction-diffusion and Hamilton-Jacobi PDE.
\end{example}

We next provide an application of our results
to a different model, first passage percolations in time-dependent environments.  Let $V_d$ be the set of edges of the lattice $\bbZ^d$, that is, each $v\in V_d$ connects two points $A,B\in\bbZ^d$ which share $d-1$ of their $d$ coordinates and differ by 1 in the last coordinate (these can be either directed edges or not).  Let us consider a traveler moving on the lattice $\bbZ^d$ from point $A$ to $B$.  
He can move along any path $\gamma$ made of a sequence of edges $v_1^\gamma, v_2^\gamma,\dots,v_{n_\gamma}^\gamma$, where each $v_i^\gamma$ connects some points $A_{i-1}$ and $A_{i}$, with $A=A_0$ and $B=A_{n_\gamma}$.  Let us denote by $\Gamma_{A,B}$ the set of all such paths.
Let us assume that the travel time for any edge $v$, if it is reached by the traveler at time $t$, is some number $\tau^t_v\ge 0$.  For any $\gamma\in \Gamma_{A,B}$ and any time $t_0$, define recursively (for $i=1,2,\dots,n_\gamma$) the times
\[
t_{i}:=t_{i-1}+ \tau^{t_{i-1}}_{v_{i}^\gamma} \qquad\text{and}\qquad T^{t_0}_\gamma:= t_{n_\gamma}-t_0.
\]
That is, $t_i$ is the time of arrival at the point $A_i$, and $T^{t_0}_\gamma$ is the travel time along $\gamma$ when the starting time is $t_0$.  Finally, let 
\beq\lb{5.1}
X^{t}(A,B):=\inf\left\{ T^t_\gamma\,\big|\, \gamma \in\Gamma_{A,B} \right\}
\eeq
be the shortest travel time from $A$ to $B$ when starting at time $t$.

When the travel times are independent of $t$, this is of course the standard first passage percolation model.  Let us consider one of the following two setups when time-dependence is included.  Let $\xi_v^t\ge 0$ be some number, and let $\tau_v^t$ be either the first time  such that 
\beq\lb{5.2}
\int_0^{\tau_v^t} \xi_v^{t+s} ds=1,
\eeq
or let 
\beq\lb{5.3}
\tau_v^t:= \inf \left\{ s+ \left( \xi_v^{t+s} \right)^{-1} \,\Big|\, s\ge 0\right\}.
\eeq
In the first case, one can think of $\xi_v^{t+s}$ as the instantaneous travel speed along $v$ at time $t+s$, which changes due to changing road conditions (so $\int_0^{\tau} \xi_v^{t+s} ds$ is distance traveled in time $\tau$).  In the second case, one can think of $\xi_v^{t+s}$ as the speed of a train leaving one end of $v$ at time $t+s$ (which could be zero if there is no such train), and the traveller chooses the one that brings him to the other end at the earliest time.  

Now for any $e\in\bbZ^d$ we can define $X_{m,n}^{t,e}:=X^t(me,ne)$, so that asymptotic speed of travel in direction $e$  is $|e|$ divided by the reciprocal of \eqref{5.4},
provided this deterministic limit exists.  Theorems \ref{T.1.1} and \ref{T.1.2} can again be used to show this, either almost surely or in probability, if the speeds $\xi_v^t$ are random variables satisfying appropriate hypotheses.  

Note that hypotheses (1) and (6) in these theorems will always be satisfied (the latter with $c:=\infty$ and any $C\ge 0$) for both models \eqref{5.2} and \eqref{5.3}.  If there is $L<\infty$ such that  for all $(t,v)\in[0,\infty)\times V_d$ we have $\int_t^{t+L}\xi_v^sds\ge 1$ or $\sup\{\xi_v^s\,|\, s\in[t,t+L]\} \ge \frac 1L$ when we define $\tau_v^t$ via \eqref{5.2} or via \eqref{5.3}, respectively, this will also  guarantee ($2^{*}$) (and so (2) as well).  Finally, we will let $\calF_t^-$ be the $\sigma$-algebra generated by the family of random variables
\beq\lb{5.5}
\{\xi_v^s \,|\, s\in[0,t] \,\&\, v\in V_d\},
\eeq
and $\calF_t^+$ the $\sigma$-algebra generated by the family of random variables
\beq\lb{5.6}
\{\xi_v^s \,|\, s\ge t \,\&\, v\in V_d\},
\eeq
which will  guarantee (4).

We note that (3) follows from space-time stationarity of $\xi_v^t$.  For any $y\in\bbZ^d$, the translation $\sigma_y(x):=x+y$ on $\bbZ^d$ induces a translation map on $V_d$, which we also call $\sigma_y$.
If $(\Omega,\calF,\bbP)$ is the involved probability space, then the speeds $\xi_v^t$ are space-time stationary provided there is a semigroup of measure-preserving bijections $\{{\Upsilon_{(s,y)}:\Omega\to\Omega}\}_{(s,y)\in[0,\infty)\times\bbZ^d}$ such that $\Upsilon_{(0,0)}={\rm Id}_\Omega$, for any $(s,y),(r,z)\in [0,\infty)\times\bbZ^d$ we have
\[
\Upsilon_{(s,y)}\circ\Upsilon_{(r,z)}=\Upsilon_{(s+r,y+z)},
\] 
and for any $(t,s,v,y,\omega)\in [0,\infty)^2\times V_d\times \bbZ^{d}\times\Omega$ we have
\[
\xi_v^t ( \Upsilon_{(s,y)} \omega)= \xi_{\sigma_y(v)}^{t+s} (\omega).
\]
Hence if the speeds $\xi_v^t$ are also space-time stationary, we will only need to check  (5) or ($5^{*}$).


We can construct space-time stationary environments with appropriately time-decreasing correlations by sampling space-stationary environments.  Let $(\Omega_0,\calF_0,\bbP_0)$ be a probability space and let $\sigma_y$ be as the above.
We say that a random field $\eta:V_d\times\Omega_0\to\bbR$ is space stationary, if there is a semigroup of measure-preserving bijections $\{{\Upsilon_{y}:\Omega_0\to\Omega_0}\}_{y\in \bbZ^d}$ such that $\Upsilon_{0}={\rm Id}_{\Omega_0}$, for any $y,z\in \bbZ^d$ we have $\Upsilon_{y}\circ\Upsilon_{z}=\Upsilon_{y+z}$,
and for any $(v,y,\omega)\in  V_d\times \bbZ^{d}\times\Omega_0$ we have
\[
\eta (v, \Upsilon_{y} \omega)= \eta (\sigma_y(v),\omega).
\]
Let us assume below that $\eta$ satisfies this as well as $\frac 1L\le \eta \le L$ for some $L\ge 1$.

\begin{example} \lb{E.5.1}
Let $\Omega:=[0,C)\times \Omega_0^{\bbN_0}$ have the product probability measure (with some $C>0$ and the uniform measure on $[0,C)$), and for $\omega=(a,\omega_0,\omega_1,\dots)\in\Omega$ let 
\beq\lb{5.7}
\xi_v^t(\omega):=\eta(v,\omega_{\lfloor (t+a)/C\rfloor}).
\eeq
That is, the speeds $\xi_v^t$ always change after time interval $C$, starting from some time in $[0,C)$.
Then they  are clearly space-time stationary. Moreover,  if $\calF_t^\pm$ are defined via \eqref{5.5} and \eqref{5.6}, then $\calF_t^-$ and $\calF_{t+C}^+$ are independent for each $t\ge 0$ because random variables $\alpha(\omega):=\eta(v_1,\omega_i)$ and $\beta(\omega):=\eta(v_2,\omega_j)$ are independent for any $v_1,v_2\in V_d$ and any distinct $i,j\in\bbN_0$.  The above discussion now shows that Theorem \ref{T.1.1} applies to $X_{m,n}^{t,e}$ above for any $e\in\bbZ^d$, so $\frac 1n X_{0,n}^{t,e}$ converges to some $\omega$-independent constant almost surely.  

Moreover, for any $(A,B,t)\in\bbZ^{2d}\times[0,\infty)$  (and with $L$ above) we clearly have
\beq\lb{5.8}
L^{-1} |A-B|_1 \le X^t(A,B) \le L|A-B|_1 ,
\eeq
where $|e|_1:=|e_1|+\dots+|e_d|$ is the $L^1$ norm, so the deterministic limit \eqref{5.4} is from $[\frac 1L|e|_1,L|e|_1]$.  Let us denote by $B_r^1(0)$ the ball in $\bbR^d$ with respect to the $L^1$ norm, with radius $r$ and centered at the origin.  Then as in Example \ref{E.4.3}, we can show that there is convex open $S\subseteq\bbR^d$, containing $B_{1/L}^1(0)$ and contained in $B_L^1(0)$, such that if $S_t(\omega)$ is the set of all $A\in\bbZ^d$ with $X^0(0,A)\le t$ (for $t\ge 0$ and $\xi_v^s$ from \eqref{5.7}), then for almost every $\omega\in\Omega$ we have that for any $\delta>0$ and all large-enough $t\ge 0$ (depending on $\omega$ and $\delta$),
\beq\lb{5.9}
(1-\delta)tS \cap\bbZ^d\subseteq S_t(\omega) \subseteq (1+\delta)tS\cap\bbZ^d.
\eeq
That is, $S$ is again the deterministic asymptotic  shape of all points reachable from the origin in time $t$ (as $t\to\infty$ and after scaling by $t$).
\end{example}

\begin{example} \lb{E.5.2}
Consider a Poisson point process with parameter $\lambda>0$ on $\bbR$, defined on some  probability space $(\Omega',\calF',\bbP')$,  and let $N_t$ be the corresponding counting process (i.e., $N_t$ is the number of points in the interval $(0,t]$).  We now let $\Omega:=\Omega'\times \Omega_0^{\bbN_0}$ have the product probability measure, and for $\omega=(\omega',\omega_0,\omega_1,\dots)\in\Omega$ we let 
\[
\xi_v^t(\omega):=\eta(v,\omega_{N_t}).
\]
That is, now the interval after which the speeds $\xi_v^t$ change has an exponential distribution.
The speeds are again space-time stationary, and ($5^*$)  holds with $\phi(s):=e^{-\lambda s}$ when $\calF_t^\pm$ are defined via \eqref{5.5} and \eqref{5.6}. 
Indeed, if $G_{t,s}:= \{ N_{t+s} = N_{t }\}$ for $t,s\ge 0$, then
$ \bbP[G_{t,s}] = e^{-\lambda s}$ and events $E$ and $F\cap G_{t,s}^c$ are independent whenever $E\in\calF_t^-$ and $F\in\calF_{t+s}^+$ 
(see below).  
This includes $F=\Omega$, which yields for general $E\in\calF_t^-$ and $F\in\calF_{t+s}^+$,
\[
0\le \bbP[F\cap G_{t,s}\cap E] \le \bbP[G_{t,s}\cap E]=\bbP[G_{t,s}] \bbP[E].
\]
Therefore $\left|\bbP[F\cap G_{t,s}|E]-P[F\cap G_{t,s}]\right| \le   \bbP[G_{t,s}] $ and so
\[
\left|\bbP[F|E]-P[F]\right|\leq \left|\bbP[F\cap G_{t,s}^c|E]-P[F\cap G_{t,s}^c]\right|+\bbP[G_{t,s}] = e^{-\lambda s}.
\]
The above discussion therefore shows that Theorem \ref{T.1.2} applies to $X_{m,n}^{t,e}$ above for any $e\in\bbZ^d$, so $\frac 1n X_{0,n}^{0,e}$ converges to some $\omega$-independent constant almost surely.  And just as before, we can again also conclude \eqref{5.8} and \eqref{5.9}.

It remains to prove independence of $E$ and $F\cap G_{t,s}^c$ for any $E\in\calF_t^-$ and $F\in\calF_{t+s}^+$.
Let us denote $v_0,v_1,\dots$ all the edges in $V_d$ and for $m,J\in\bbN_0$ let  $Y_{m}^J(\omega):= (\eta(v_0,\omega_{m}),\dots,\eta(v_J,\omega_{m}) )$.
By Dynkin's $\pi$-$\lambda$ Theorem, it suffices to show that $\bbP[E\cap F\cap G_{t,s}^c]=\bbP[E]\bbP[ F\cap G_{t,s}^c]$ for
\[
E= \left\{ Y_{N_{t_i}}^J \in A_{i}\text{ for }i=1,\dots, n \right\} \qquad\text{and}\qquad F= \left\{Y_{N_{t_{i}}}^J \in A_{i} \text{ for } i=n+1, \dots, 2n \right\},
\]
with arbitrary $J\in\bbN_0$, Borel sets  $A_{1},\dots,A_{2n} \subseteq \bbR^J$, and times 
\[
0\le t_1<\dots<t_n=t<t+s=t_{n+1}<\dots<t_{2n}.
\]
Note that $N_{t_{i}}\ge N_{t_{i-1}}$ for all $i$ (let $t_0:=0$, so $N_{t_0}\equiv0$), and for any $k_1,\dots,k_{2n}\in\bbN_0$ we have
\[
\bbP\left[ N_{t_{i}}-N_{t_{i-1}}=k_i \text{ for } i=1,\dots,2n \right]= \prod_{i=1}^{2n} \frac{ (\lambda (t_i-t_{i-1}))^{k_i}}{k_i!}e^{-\lambda(t_i-t_{i-1})} =: \prod_{i=1}^{2n} p_{i,k_i}
\]
(clearly $\sum_{k\in\bbN_0} p_{i,k}=1$).  Since $G_{t,s}^c=\{N_{t_{n+1}}>N_{t_n}\}$,  with $K_2:=(k_{n+1},\dots,k_{2n})$ 
we obtain
\begin{align*}
\bbP \left[ F\cap G_{t,s}^c\right]&=
\sum_{ K_2 \in \bbN\times\bbN_0^{n-1}}
\bbP \left[ Y_{N_{t_n} +\sum_{j=n+1}^i k_i}^J \in A_i \,\&\, N_{t_{i}}-N_{t_{i-1}}=k_i  \text{ for } i=n+1,\dots,2n \right]\\
&= \sum_{ K_2 \in \bbN\times\bbN_0^{n-1}} \left( \prod_{i=n+1}^{2n} p_{i,k_i} \right)
\bbP \left[ Y_{\sum_{j=n+1}^i k_i}^J \in A_i  \text{ for } i=n+1,\dots,2n \right]
\end{align*}
because the $\sigma$-algebras $\calF'\times\{\emptyset,\Omega_0^{\bbN_0}\}$ and  $\{\emptyset,\Omega'\} \times \calF_0^{\bbN_0}$ are independent, random variables $\{N_{t_i}-N_{t_{i-1}}\}_{i=1,\dots,2n}$ are jointly independent,  and the joint distribution of $\{Y_{m}^J, Y_{m+1}^J, \dots \}$ is independent of $m$.  
But then with $K_1:=(k_{1},\dots,k_{n})$ we similarly  obtain the desired claim
\begin{align*}
\bbP \left[ E\cap F\cap G_{t,s}^c\right]&=
\sum_{ (K_2,K_1) \in \bbN\times\bbN_0^{2n-1}}
\bbP \left[ Y_{\sum_{j=1}^i k_i}^J \in A_i \,\&\, N_{t_{i}}-N_{t_{i-1}}=k_i  \text{ for } i=1,\dots,2n \right]\\
&= \sum_{ (K_2,K_1) \in \bbN\times\bbN_0^{2n-1}} \left( \prod_{i=1}^{2n} p_{i,k_i} \right)
\bbP \left[ Y_{\sum_{j=1}^i k_i}^J \in A_i  \text{ for } i=1,\dots,2n \right] \\
&= \sum_{ K_1 \in \bbN_0^{n}} \left( \prod_{i=1}^{n} p_{i,k_i} \right)
\bbP \left[ Y_{\sum_{j=1}^i k_i}^J \in A_i  \text{ for } i=1,\dots,n \right] \\
&  \qquad\qquad \sum_{ K_2 \in \bbN\times\bbN_0^{n-1}} \left( \prod_{i=n+1}^{2n} p_{i,k_i} \right)
\bbP \left[ Y_{\sum_{j=1}^i k_i}^J \in A_i  \text{ for } i=1,\dots,2n \right] \\
&= \sum_{ K_1 \in \bbN_0^{n}} \left( \prod_{i=1}^{n} p_{i,k_i} \right)
\bbP \left[ Y_{\sum_{j=1}^i k_i}^J \in A_i  \text{ for } i=1,\dots,n \right] \\
&  \qquad\qquad \sum_{ K_2 \in \bbN\times\bbN_0^{n-1}} \left( \prod_{i=n+1}^{2n} p_{i,k_i} \right)
\bbP \left[ Y_{\sum_{j=n+1}^i k_i}^J \in A_i  \text{ for } i=n+1,\dots,2n \right] \\
& =\bbP[E] \, \bbP \left[F\cap G_{t,s}^c\right],
\end{align*}
where  we also used $k_{n+1}\ge 1$ in the third equality.

\end{example}




\begin{thebibliography}{1}

\bibitem{BIN}
{\sc D.~Burago, S.~Ivanov, and A.~Novikov}, 
{\em Feeble fish in time-dependent waters and homogenization of the $g$-equation},  
Comm. Pure Appl. Math. {\bf 73} (2020), 1453--1489.
  
\bibitem{CoxDur}
J.T.~Cox and R.~Durrett,
{\it Some limit theorems for percolation processes with necessary and sufficient conditions},
Ann. Probab. {\bf 9} (1981), 583--603. 

\bibitem{KW}
{\sc Y.~Katznelson and B.~Weiss}, 
{\em A simple proof of some ergodic  theorems}, 
 Israel J. Math.  {\bf 42} (1982), 291--296.
  
  
\bibitem{kingman}
{\sc J.~F. Kingman}, 
{\em The ergodic theory of subadditive stochastic  processes}, 
 J. Roy. Statist. Soc. Ser. B {\bf 30} (1968), 499--510.

\bibitem{lev}
{\sc S.~Levental}, 
{\em A proof of Liggett's version of the subadditive ergodic  theorem}, 
 Proc. Amer. Math. Soc. {\bf 102}  (1988), 169--173.
  
  \bibitem{Lig} 
  T.M.~Liggett,
\it An improved subadditive ergodic theorem,
\rm Ann. Probab. {\bf 13} (1985), 1279--1285.

  \bibitem{Ric}
D.~Richardson, 
{\it Random growth in a tessellation},
 Proc. Cambridge Philos. Soc. {\bf 74} (1973), 515--528. 

 \bibitem{ZhaZla5} Y.P. Zhang and A.~Zlato\v s,
\it Homogenization for space-time-dependent KPP reaction-diffusion equations and G-equations,
\rm preprint.

\end{thebibliography}

\end{document}